\documentclass[EJP]{ejpecp} 

\usepackage{stmaryrd}
\usepackage{wasysym}
\usepackage{hyperref}
\usepackage{cleveref}
\usepackage{enumitem}
\usepackage{dsfont}
\usepackage{soul}
\hypersetup{colorlinks=true,linkcolor=black,citecolor=black}
\usepackage{color}
\usepackage[all]{xy}
\usepackage{float}
\usepackage{bm}
\usepackage{mathtools}
\usepackage{thmtools}
\usepackage{setspace}
\usepackage{comment}
\usepackage{tikz}
\usepackage{mathdots}
\usepackage[myheadings]{fullpage}

\SHORTTITLE{$q$-TASEP with position-dependent slowing}

\TITLE{$q$-TASEP with position-dependent slowing} 

\AUTHORS{%
  Roger Van Peski\footnote{Massachusetts Institute of Technology, United States of America.
    \EMAIL{rvp@mit.edu}}}


\AMSSUBJ{60K35; 05E05} 

\SUBMITTED{January 7, 2022} 
\ACCEPTED{November 2, 2022} 

\ARXIVID{2112.03725} 

\VOLUME{0}
\YEAR{2020}
\PAPERNUM{0}
\DOI{10.1214/YY-TN}

\ABSTRACT{We introduce a new interacting particle system on $\Z$, \emph{slowed $t$-TASEP}. It may be viewed as a $q$-TASEP with additional position-dependent slowing of jump rates, depending on a parameter $t$, which leads to discrete asymptotic fluctuations at large time. If on the other hand $t \to 1$ as $\text{time} \to \infty$, we prove
\begin{enumerate}
    \item A law of large numbers for particle positions,
    \item A central limit theorem, with convergence to the fixed-time Gaussian marginal of a stationary solution to SDEs derived from the particle jump rates, and
    \item A bulk limit to a certain explicit stationary Gaussian process on $\R$, with scaling exponents characteristic of the Edwards-Wilkinson universality class in $(1+1)$ dimensions. 
\end{enumerate}
The proofs relate slowed $t$-TASEP to a certain Hall-Littlewood process, and use contour integral formulas for observables of this process. Unlike most previously studied Macdonald processes, this one involves only local interactions, resulting in asymptotics characteristic of $(1+1)$-dimensional rather than $(2+1)$-dimensional systems.}

\numberwithin{equation}{section}

\makeatletter
\def\@tocline#1#2#3#4#5#6#7{\relax
  \ifnum #1>\c@tocdepth 
  \else
    \par \addpenalty\@secpenalty\addvspace{#2}%
    \begingroup \hyphenpenalty\@M
    \@ifempty{#4}{%
      \@tempdima\csname r@tocindent\number#1\endcsname\relax
    }{%
      \@tempdima#4\relax
    }%
    \parindent\z@ \leftskip#3\relax \advance\leftskip\@tempdima\relax
    \rightskip\@pnumwidth plus4em \parfillskip-\@pnumwidth
    #5\leavevmode\hskip-\@tempdima
      \ifcase #1
       \or\or \hskip 1em \or \hskip 2em \else \hskip 3em \fi%
      #6\nobreak\relax
    \hfill\hbox to\@pnumwidth{\@tocpagenum{#7}}\par
    \nobreak
    \endgroup
  \fi}
\makeatother

\DeclareSymbolFont{bbold}{U}{bbold}{m}{n}
\DeclareSymbolFontAlphabet{\mathbbold}{bbold}

\DeclarePairedDelimiter{\abs}{\lvert}{\rvert}
\newcommand{\floor}[1]{\lfloor #1 \rfloor}
\newcommand{\ceil}[1]{\lceil #1 \rceil}
\newcommand{\R}{\mathbb{R}}
\newcommand{\Z}{\mathbb{Z}}

\newcommand{\N}{\mathbb{N}}
\newcommand{\C}{\mathbb{C}}

\newcommand{\E}{\mathbb{E}}

\newcommand{\mf}{\mathfrak}

\newcommand{\bbone}{\mathbbold{1}}
\renewcommand{\l}{\lambda}
\newcommand{\eps}{\epsilon}
\renewcommand{\Re}[1]{\text{Re}(#1)}
\renewcommand{\Re}{\operatorname{Re}}

\newcommand{\var}{\text{Var}}
\newcommand{\pfrac}[2]{\left(\frac{#1}{#2}\right)}

\newcommand{\dde}[3]{\left. \frac{d #1}{d #2} \right|_{#3}} 

\newcommand{\dderiv}[2]{\frac{d #1}{d #2}}

\newcommand{\ot}{\otimes}

\newcommand{\tz}{{\tilde{z}}}
\newcommand{\tw}{{\tilde{w}}}

\newcommand{\la}{\left\langle}
\newcommand{\ra}{\right\rangle}

\newcommand{\tth}{^{th}}

\renewcommand{\L}{\Lambda}
\newcommand{\Y}{\mathbb{Y}}

\newcommand{\bx}{\mathbf{x}}
\newcommand{\by}{\mathbf{y}}
\newcommand{\bi}{\mathbf{i}}
\newcommand{\cP}{\mathbb{Y}}

\DeclareMathOperator{\len}{len}

\DeclareMathOperator{\Cov}{Cov}
\DeclareMathOperator{\const}{const}

\newcommand{\gap}{\operatorname{gap}}

\newcommand{\teps}{\tilde{\epsilon}}

\begin{document}


\tableofcontents

\section{Introduction}

\subsection{The model and main asymptotic results.}

Consider a configuration of particles on $\Z$ at some positions $x_1 > x_2 > \cdots$, at most one particle per site, evolving in continuous time. Each particle has an independent Poisson clock and jumps $1$ unit to the right whenever it rings. The clock of the $i\tth$ particle from the right has rate $1-q^{x_{i-1}-x_i-1}$, often simply written $1-q^{\gap}$, where $0 \leq q < 1$ and we take $x_0 := \infty$. This is the well-known \emph{$q$-TASEP}, introduced in \cite{borodin2014macdonald}, which reduces to the usual \emph{totally asymmetric exclusion process (TASEP)} when $q=0$. The asymptotics of $q$-TASEP and its relatives in various regimes have been the subject of much recent work, for example \cite{barraquand2015phase,borodin2014duality,borodin2015discrete,ferrari2015tracy,orr2017stochastic,imamura2019q,imamura2019fluctuations,vetHo2021asymptotic}. These asymptotics crucially rely on the exact solvability of the model, which derives from its connection to Macdonald processes \cite{borodin2014macdonald}.

An inhomogeneous version of $q$-TASEP, where the $i\tth$ particle has jump rate $a_i(1-q^{\gap})$ for some fixed positive real parameters $a_1,a_2,\ldots$, was introduced simultaneously in \cite{borodin2014macdonald}. Such inhomogeneities often yield different asymptotic behaviors: for instance, \cite{barraquand2015phase} showed that by tuning the $a_i$ correctly, one may see the Baik-Ben Arous-Peche distributions in the limit, generalizing Tracy-Widom asymptotics established in \cite{ferrari2015tracy}.

In this work, we introduce and analyze an integrable particle system, which may be viewed as a $q$-TASEP with inhomogeneous jump rates, which depend not only on the particles but on their positions. Explicitly, the $i\tth$ particle from the right, which sits at $x_i$, has jump rate $t^{x_i+i}(1-t^{\gap})$ for a parameter\footnote{We have changed notation from usual $q$-TASEP because this $t$ corresponds to the $t$ in Hall-Littlewood polynomials, see below.} $t \in (0,1)$. Hence each particle has a base jump rate $t^i$ which is slower for particles further behind the leading particle, but also has a position-dependent slowing $t^{x_i}$ which causes it to slow down as it moves further to the right. We refer to this system as \emph{slowed $t$-TASEP}. Like $q$-TASEP, slowed $t$-TASEP is related to Macdonald processes---specifically, the special case of Hall-Littlewood processes, see next subsection---and from this connection derives certain exact formulas which allow fine asymptotic analysis.

The position-dependent damping means that slowed $t$-TASEP behaves quite differently, as is already apparent with the rightmost particle. In $q$-TASEP, this particle jumps according to a Poisson process with rate $1$, hence has asymptotically $(\text{time})^{1/2}$-order Gaussian fluctuations. By contrast, in slowed $t$-TASEP the rightmost particle's jump rate $t^{x_1}$ decreases by a factor of $t$ each step, so while its position still goes to $\infty$ as $\text{time} \to \infty$, one expects its fluctuations have bounded order, which we show in \Cref{sec:fixed_t} for normalized exponential transforms of particle positions $t^{-(x_1 - \text{const}(\tau))}$ and similarly for other particles. Hence, while particle positions in $q$-TASEP with fixed $q$ exhibit Gaussian fluctuations (for a fixed particle, as above) and Tracy-Widom asymptotics (when particle index is scaled along with time, see \cite{ferrari2015tracy}), such continuous probability distributions will not appear in slowed $t$-TASEP with $t$ fixed.

We therefore study the regime in which $\text{time} \to \infty$ and $t \to 1$ simultaneously, which ameliorates but does not obliterate the position-dependent slowing. Our first asymptotic result is that the position of each particle obeys an explicit law of large numbers in this regime. 

\begin{restatable}{theorem}{llnintro}\label{thm:LLN_intro}
Let $(x_1(T),x_2(T),\ldots)$ be the particle positions in slowed $t$-TASEP after time $T$, with $t=e^{-\eps}$ and packed initial condition $x_k(0) = -k$. Then for any $\tau > 0$ and $k \in \Z_{>0}$,
\begin{equation*}
    \eps \cdot x_k(\tau/\eps) \to \log \left(\sum_{j=0}^k \frac{\tau^j}{j!}\right) - \log \left(\sum_{j=0}^{k-1} \frac{\tau^j}{j!} \right) \quad \quad \text{in probability as $\eps \to 0^+$.}
\end{equation*}
\end{restatable}

In particular, particles become macroscopically far apart in the limit. In simulations with fixed $t \approx 1$, one may observe the first particle `peeling off' from the bulk while the second particle barely moves at all due to the $1-t^{\gap}$ component of its jump rate until the gap becomes large. Then the second particle `peels off', and once it is far away the third begins to move nontrivially, etc.

However, since particles affect those behind them due to this $1-t^{\gap}$ factor in the jump rates, despite the macroscopic separation they continue to influence one another at the level of fluctuations. In the scaling of time and $t$ of \Cref{thm:LLN_intro}, we have that the rescaled fluctuations
\[
\eps^{1/2} \left(x_k(\tau/\eps) - \E\left[x_k(\tau/\eps)\right]\right)
\]
converge to Gaussians $X^{(k)}_\tau$ with nontrivial covariances determined by an $(r+s)$-fold contour integral formula for 
\[
\Cov(X^{(1)}_\tau+\ldots+X^{(r)}_\tau,X^{(1)}_\tau+\ldots+X^{(s)}_\tau),
\]
see \Cref{thm:gaussianity}. These limiting covariances still depend on $\tau$, but converge without rescaling as $\tau \to \infty$, a manifestation of the believed convergence to a stationary distribution in the prelimit particle system which is discussed in \Cref{sec:fixed_t}.

\begin{theorem}\label{thm:auxlimit_intro}
As $\tau \to  \infty$, the random variables $X^{(i)}_\tau$ converge in distribution to the fixed-time marginal of the unique stationary solution $(Z^{(1)}_T,Z^{(2)}_T,\ldots)$ to the system
\begin{equation}\label{eq:SDEs_intro}
    dZ^{(k)}_T = \left((k-1)Z^{(k-1)}_T - k Z^{(k)}_T\right)dT + dW^{(k)}_T \quad \quad k=1,2,\ldots
\end{equation}
where $W^{(k)}_T$ are independent standard Brownian motions. Their covariances furthermore have the explicit form 
\begin{equation*}
    \Cov(Z^{(r)}_T,Z^{(s)}_T) = \frac{1}{4 \pi^2} \displaystyle \oint_{\Gamma_0} \oint_{\Gamma_{0,w}} \frac{w}{z-w} \frac{r! s!}{z^{r}w^{s}}e^{z+w}(1-z/r)(1-w/s)\frac{dz}{z} \frac{dw}{w}
\end{equation*}
with the $w$-contour enclosing $0$ and enclosed by the $z$-contour.
\end{theorem}

In addition to reflecting prelimit convergence to stationarity, \Cref{thm:auxlimit_intro} yields a $2$-fold rather than $(r+s)$-fold contour integral formula, which allows analysis in the bulk regime $r,s \to \infty$. The proof of this reduction of covariance formulas is by orthogonal polynomial methods inspired by the similar arguments of \cite[\S 5.1]{borodin2018anisotropic}, though the interpretation as a stationary solution to a system of SDEs is not present there.

A natural way to study the bulk limit of slowed $t$-TASEP is to string the $Z^{(k)}_0$, which represent asymptotic fluctuations of particle positions, together into a stochastic process $Y_T, T \in \R_{>0}$ by linear interpolation. Explicitly, set $Y_0=0$, $Y_T = Z^{(T)}_0$ when $T \in \Z_{>0}$, and linearly interpolate times between these, see \Cref{fig:Ygraph}. The bulk limit is then encoded by the scaling limit of $Y_T$ for large $T$, which we explicitly compute by taking asymptotics of the covariance formula in \Cref{thm:auxlimit_intro} via steepest descent. 

\bigskip

\begin{figure}[htbp]
\centering
\includegraphics{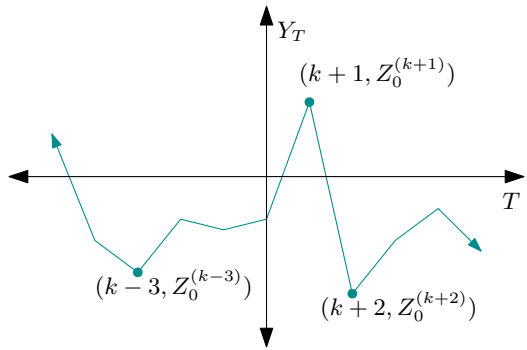}
\caption{The graph of the process $Y_T$, which is just a piecewise linear interpolation from the (random) points $(T,Z^{(T)}_0)$, shown in a window around $T=k$. 
}\label{fig:Ygraph}
\end{figure}

\bigskip

\begin{restatable}{theorem}{bulkintro}\label{thm:bulk_intro}
The process 
\[
R^{(T)}_s := T^{1/4} Y_{T+s \sqrt{T}}
\]
converges in finite-dimensional distributions as $T \to \infty$ to the unique stationary Gaussian process $R_s, s \in \R$ with covariances
\begin{equation*}
    \Cov(R_{a}, R_{b}) =\int_0^\infty y^2 e^{-y^2-|b-a|y}dy.
\end{equation*}
\end{restatable}

The scaling exponents in \Cref{thm:bulk_intro} are characteristic of the \emph{Edwards-Wilkinson universality class} in $(1+1)$ dimensions ($1$ spatial dimension plus time), see \cite{seppalainen2010current} for other examples of interacting particle systems in this class. The specific integral form of the covariance is somewhat similar to, but not the same as, covariances for solutions to the $(1+1)$-dimensional additive stochastic heat equation, see for example \cite[\S2.3.2]{hairer2009introduction}. We suspect it may arise from some transform of solutions to this or a similar stochastic PDE, but do not have any results in this direction. However, the fact that the limiting fluctuations are described by a $1$-dimensional Gaussian process of any kind is surprising given the algebraic origins of slowed $t$-TASEP, which we discuss next.

\subsection{Hall-Littlewood processes.}

The main tool in proving the above asymptotic results is an explicit contour integral formula for $t$-moments of the particle positions, see \Cref{thm:observable_formula}. This arises because slowed $t$-TASEP is essentially a reparametrization of a certain \emph{Hall-Littlewood process}, for which such formulas were proven in \cite{bufetov2018hall}. These are a special case of the \emph{Macdonald processes} introduced in \cite{borodin2014macdonald}. 

Relevant definitions and results will be recalled in \Cref{sec:prelim}, but briefly, the (skew) Hall-Littlewood polynomials $P_\l(x_1,\ldots,x_n;t)$ and $Q_{\nu/\l}(x_1,\ldots,x_n;t)$ are a special class of symmetric polynomials in the variables $x_i$, indexed by integer partitions $\l$, which feature an extra parameter $t$ which we take to be in $(0,1)$. The Hall-Littlewood process we consider is a continuous-time, Markovian process $\l(\tau)$ on the set $\Y_n$ of integer partitions of length $\leq n$, determined by 
\begin{equation}\label{eq:hlproc_intro}
    \Pr(\l(\tau_0+\tau)=\nu | \l(\tau_0) = \mu) \propto \left(\lim_{D \to \infty} Q_{\nu/\mu}\left(\underbrace{\tau/D,\ldots,\tau/D}_{D\text{ times}};t\right) \right) \frac{P_\nu(1,t,t^2,\ldots,t^{n-1};t)}{P_\mu(1,t,t^2,\ldots,t^{n-1};t)}
\end{equation}
for $\mu,\nu \in \Y_n$ and initial condition $\l(0) = (0,\ldots,0)$ (the limit of skew $Q_{\nu/\mu}$ is known as a Plancherel specialization, see \Cref{sec:sampling}). The dynamics of \eqref{eq:hlproc_intro} also make sense when the finite geometric progression $1,t,\ldots,t^{n-1}$ is replaced with an infinite one, yielding dynamics on the set of all partitions $\Y$. An integer partition $\l = (\l_1,\l_2,\ldots)$ may be identified with a particle configuration on $\Z$ by taking the position of the $k\tth$ particle from the right to be
\begin{equation}\label{eq:coord_change}
   x_k = \l_k - k 
\end{equation}
for $k \geq 1$. When the conjugate partition $\l'$ evolves according to the above Hall-Littlewood process with $n = \infty$, the particle positions defined in this way evolve according to slowed $t$-TASEP, see \Cref{thm:HL_qTASEP_connection}. 

\subsection{Discussion: Macdonald processes, locality, and dynamics in $(1+1)$ and $(2+1)$ dimensions.}

The dynamics on partitions described above appear naturally as marginals of dynamics on triangular arrays of integers, or \emph{Gelfand-Tsetlin patterns}, which are simply sequences $\l^{(n)} \in \Y_n, n \geq 1$ which \emph{interlace} in the sense that \[\l^{(n+1)}_1 \geq \l^{(n)}_1 \geq \l^{(n+1)}_2 \geq \ldots \geq \l^{(n+1)}_{n+1}.\]
These are often visualized as infinite configurations of particles in the plane by placing a particle at each point $(\l^{(n)}_i, n)$ as in \Cref{fig:particle_array} (middle).

\bigskip

\begin{figure}[htbp]
\centering 
\includegraphics{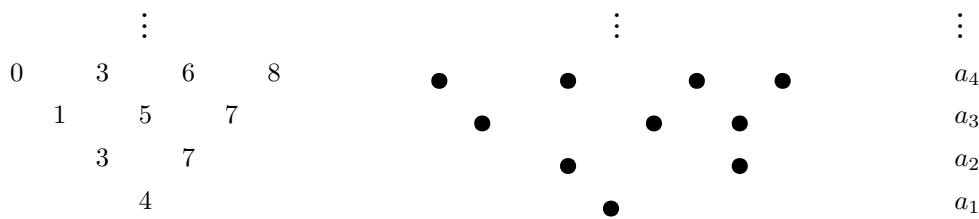}
\caption{The bottom four rows of an infinite Gelfand-Tsetlin pattern, visualized as a sequence of partitions/array of integers (left) and as a particle configuration (middle), with jump rates $a_i$ for the Poisson clocks of each row (right).}\label{fig:particle_array}
\end{figure}

\bigskip

Simple continuous-time dynamics on such arrays, such that the $n\tth$ row evolves by the Hall-Littlewood process dynamics of \eqref{eq:hlproc_intro} for each $n$, were given in \cite[\S6]{borodin2016between}. In such dynamics, the $n\tth$ row has a Poisson clock of rate $a_i$ for every $i$, where in our setting $a_i = t^{i-1}$. When a row's clock rings, one of the particles in that row jumps, specifically the leftmost one whose jump would not violate interlacing with the row below. This jump in turn triggers a particle in the row above to jump by $1$ according to certain rules, which triggers one in the row above that, etc., in such a way that interlacing is preserved and the triggering of moves is local.

At least in special cases, such dynamics also exist for more general \emph{Macdonald processes}, given by replacing the Hall-Littlewood polynomials in \eqref{eq:hlproc_intro} by Macdonald polynomials $P_\l(x_1,\ldots,x_n;q,t)$ which depend on two parameters $q,t \in [0,1)$ and reduce to Hall-Littlewood polynomials when $q=0$. Such dynamics are still local in that particles' jump rates are only affected by the particles corresponding to adjacent entries of the original Gelfand-Tsetlin pattern, and were studied in depth in \cite{borodin2016nearest}. For the Schur ($q=t$) and $q$-Whittaker ($t=0$) cases with Poisson rates $a_i \equiv 1$, it was shown in \cite{borodin2014anisotropic} and \cite{borodin2018anisotropic} that global bulk asymptotics of such arrays are governed by the $2$-dimensional Gaussian free field, which exhibits logarithmic correlations. We mention also the related works \cite{borodin2009anisotropic,borodin2017stochastic} dealing with similar asymptotics of so-called $(2+1)$-dimensional growth models (i.e. growth models in two spatial and one time dimension). This is in marked contrast to the correlations $\Cov(R_s,R_{s+d})$ in the bulk limit \Cref{thm:bulk_intro}, which decay like $\text{const} \cdot |d|^{-3}$ for large $d$ and do not diverge for small $d$.  

Why this difference? \emph{The surprising feature of the Hall-Littlewood case is that not only are the dynamics on $2$-dimensional Gelfand-Tsetlin patterns governed by local interactions, but their projection to a given row of the Gelfand-Tsetlin pattern results in $(1+1)$-dimensional dynamics with only local interactions. }After the transform \eqref{eq:coord_change} this becomes the statement, visible in the definition of slowed $t$-TASEP above, that a particle's jump rate is independent of the other particles except for the one in front of it. Strictly speaking, projecting to a row of the Gelfand-Tsetlin pattern corresponds to the finite $n$ version of \eqref{eq:hlproc_intro}, which does not yield the full slowed $t$-TASEP. However, the $n=\infty$ case of \eqref{eq:hlproc_intro}, which corresponds to the full slowed $t$-TASEP, can interpreted as the projection of Hall-Littlewood process dynamics to a row `at infinity'. A fuller account is given in \cite{vanpeski2021halllittlewood}, but the basic idea is that with the initial condition where every entry of the Gelfand-Tsetlin pattern is $0$, at each time all sufficiently high rows of the Gelfand-Tsetlin pattern will yield the same partition, and the projection of the dynamics to this partition is Markovian and yields the $n=\infty$ case of \eqref{eq:hlproc_intro}. For a precise formulation of this statement in terms of the boundary of a branching graph, see \cite[Appendix A]{vanpeski2021halllittlewood}.

The locality of these dynamics on rows of the Gelfand-Tsetlin pattern is quite special to the Hall-Littlewood case, and does not hold for the general Macdonald dynamics above. Even in the $q=t$ Schur case, which usually is the simplest case of Macdonald processes, a given row of the continuous-time dynamics evolves as $n$ independent Poisson random walks conditioned in the sense of Doob $h$-transform not to intersect for all time. It therefore has highly nonlocal interactions, see \cite{borodin2013markov}. In light of this, it makes sense that the asymptotics of slowed $t$-TASEP are characteristic of $(1+1)$-dimensional growth models, while the asymptotics observed in e.g. \cite{borodin2014anisotropic,borodin2018anisotropic} are characteristic of $(2+1)$-dimensional models. Thus the apparent dissonance between our results and those discussed above is explained by the unusual locality of interactions of Hall-Littlewood processes\footnote{Let us clarify a point of potential confusion, which is that ordinary $q$-TASEP features only local interactions but arises as the projection to the leftmost particles in the array (see \Cref{fig:particle_array}) of the above-mentioned dynamics on $q$-Whittaker processes. The difference is that this projection to a $(1+1)$-dimensional system with local interactions is special and occurs only at the edge of the array, while in our Hall-Littlewood case projecting to any row yields only local interactions. It should in fact not be difficult to obtain the same asymptotics as we do in the bulk for Hall-Littlewood dynamics on Gelfand-Tsetlin patterns by considering the dynamics of \eqref{eq:hlproc_intro} with finite $n$ taken to infinity sufficiently fast along with time, which is very different from the Gaussian free field type asymptotics present in e.g. the $q$-Whittaker case \cite{borodin2018anisotropic}.} with one principal specialization $1,t,\ldots,t^{n-1}$. We mention also that this locality, for a slightly different Hall-Littlewood process, was previously exploited in \cite{van2020limits} in the context of $p$-adic random matrix theory.

\subsection{Outline.} In \Cref{sec:prelim} we introduce the necessary definitions from symmetric functions and Hall-Littlewood processes. In \Cref{sec:sampling} we introduce discrete-time Markovian dynamics on the boundary, and prove that their continuous-time Poisson limit is equivalent to slowed $t$-TASEP. In \Cref{sec:q_moments} we state a contour integral formula for observables of this process. We use these in \Cref{sec:fixed_t} to justify finiteness of fluctuations at fixed $t$, and in \Cref{sec:particle_LLN} to prove the law of large numbers \Cref{thm:LLN_intro} as $t \to 1$. In \Cref{sec:fluctuations} we show Gaussian fluctuations, and the long-time simplification of covariances \Cref{thm:aux_limit_and_nicer_covariances} which is half of \Cref{thm:auxlimit_intro}. The probabilistic justification of this additional limit via the SDEs in \Cref{thm:auxlimit_intro} is shown in \Cref{sec:SDEs}. In \Cref{sec:bulk} we prove the bulk limit to the Gaussian process given in \Cref{thm:bulk_intro}.

\subsection*{Acknowledgements}

 I am deeply grateful to Alexei Borodin for encouragement to study the main object of this work (which is a continuous-time limit of the particle system introduced in \cite{van2020limits}), for invaluable guidance on how to do so, and for feedback on the paper. I also thank Andrew Ahn for helpful conversations regarding analysis of contour integrals, and Ivan Corwin for useful comments. Finally, I thank the anonymous referees for detailed and helpful suggestions. This material is based on work partially supported by an NSF Graduate Research Fellowship under grant \#$1745302$, and by the NSF FRG grant DMS-1664619.

\section{Hall-Littlewood process preliminaries}\label{sec:prelim}

In this section we give basic definitions of symmetric functions, Hall-Littlewood polynomials, Hall-Littlewood processes, and associated Markov evolutions. For a more detailed introduction to symmetric functions see \cite{mac}, and for Macdonald processes see \cite{borodin2014macdonald}.

\subsection{Partitions, symmetric functions, and Hall-Littlewood processes.}

We denote by $\cP$ the set of all integer partitions $(\l_1,\l_2,\ldots)$, i.e. sequences of nonnegative integers $\l_1 \geq \l_2 \geq \cdots$ which are eventually $0$. We call the integers $\l_i$ the \emph{parts} of $\l$, set $\l_i' = \#\{j: \l_j \geq i\}$, and write $m_i(\l) = \#\{j: \l_j = i\} = \l_i'-\l_{i+1}'$. We write $\len(\l)$ for the number of nonzero parts, and denote the set of partitions of length $\leq n$ by $\cP_n$. We write $\mu \prec \l$ or $\l \succ \mu$ if $\l_1 \geq \mu_1 \geq \l_2 \geq \mu_2 \geq \cdots$, and refer to this condition as \emph{interlacing}. Finally, we denote the partition with all parts equal to zero by $\emptyset$.

We denote by $\L_n$ the ring $\C[x_1,\ldots,x_n]^{S_n}$ of symmetric polynomials in $n$ variables $x_1,\ldots,x_n$. It is a very classical fact that the power sum symmetric polynomials $p_k(x_1,\ldots,x_n) = \sum_{i=1}^n x_i^k, k =1,\ldots,n$, are algebraically independent and algebraically generate $\L_n$. For a symmetric polynomial $f$, we will often write $f(\bx)$ for $f(x_1,\ldots,x_n)$ when the number of variables is clear from context. We will also use the shorthand $\bx^\l := x_1^{\l_1} x_2^{\l_2} \cdots x_n^{\l_n}$ for $\l \in \cP_n$. 

One has a chain of maps
\[
\cdots \to \L_{n+1} \to \L_n \to \L_{n-1} \to \cdots \to 0
\]
where the map $\L_{n+1} \to \L_n$ is given by setting $x_{n+1}$ to $0$. 
In fact, writing $\L_n^{(d)}$ for symmetric polynomials in $n$ variables of total degree $d$, one has 
\[
\cdots \to \L_{n+1}^{(d)} \to \L_n^{(d)} \to \L_{n-1}^{(d)} \to \cdots \to 0
\]
with the same maps. The inverse limit $\L^{(d)}$ of these systems may be viewed as symmetric polynomials of degree $d$ in infinitely many variables. From the ring structure on each $\L_n$ one gets a natural ring structure on $\L := \bigoplus_{d \geq 0} \L^{(d)}$, and we call this the \emph{ring of symmetric functions}. 
An equivalent definition is $\Lambda := \C[p_1,p_2,\ldots]$ where $p_i$ are indeterminates; under the natural map $\Lambda \to \Lambda_n$ one has $p_i \mapsto p_i(x_1,\ldots,x_n)$. 

Each ring $\L_n$ has a natural basis $\{p_\l: \l_1 \leq n\}$ where
\begin{equation*}
    p_\l := \prod_{i \geq 1} p_{\l_i}.
\end{equation*}
Another natural basis, with the same index set, is given by the \emph{Hall-Littlewood polynomials}. Recall the $q$-Pochhammer symbol $(a;q)_n := \prod_{i=0}^{n-1} (1-aq^i)$, and define
\begin{equation*}
    v_\l(t) = \prod_{i \in \Z} \frac{(t;t)_{m_i(\l)}}{(1-t)^{m_i(\l)}}.
\end{equation*}

\begin{definition}\label{def:HL}
The Hall-Littlewood polynomial indexed by $\l \in \cP_n$ is
\begin{equation}\label{eq:hlP_formula}
    P_\l(\bx;t) = \frac{1}{v_\l(t)} \sum_{\sigma \in S_n} \sigma\left(\bx^\l \prod_{1 \leq i < j \leq n} \frac{x_i-tx_j}{x_i-x_j}\right)
\end{equation}
where $\sigma$ acts by permuting the variables. We often drop the `$;t$' when clear from context.
\end{definition}

It follows from the definition that $P_\l(x_1,\ldots,x_n,0) = P_\l(x_1,\ldots,x_n)$, hence for each $\l \in \cP$ there is a \emph{Hall-Littlewood symmetric function} $P_\l \in \L$.

\begin{definition}
For $\l \in \Y$, we define the dual Hall-Littlewood polynomial by
\[
Q_\l(\bx;t) = \prod_{i > 0} (t;t)_{m_i(\l)} P_\l(\bx;t).
\]
These similarly are consistent under maps $\L_{n+1} \to \L_n$ and hence define symmetric functions.
\end{definition}


Because the $P_\l$ form a basis for the vector space of symmetric polynomials in $n$ variables, there exist symmetric polynomials $P_{\l/\mu}(x_1,\ldots,x_{n-k};t) \in \L_{n-k}$ indexed by $\l \in \Y_n, \mu \in \Y_k$ which are defined by
\begin{equation*}\label{eq:def_skewP}
    P_\l(x_1,\ldots,x_n;t) = \sum_{\mu \in \Y_k} P_{\l/\mu}(x_{k+1},\ldots,x_n;t) P_\mu(x_1,\ldots,x_k;t).
\end{equation*}
The definition of $Q_{\l/\mu}$ is exactly analogous. As with non-skew Hall-Littlewood polynomials, the skew versions are consistent under the maps $\L_{n+1} \to \L_n$ and hence define symmetric functions in $\L$.

\begin{definition}\label{def:psi_varphi_coefs}
For $\mu,\l \in \Y$ with $\mu \succ \l$, let
\begin{equation*}\label{eq:pbranch}
    \psi_{\mu/\l} :=  \prod_{\substack{i > 0\\ m_i(\l) = m_i(\mu)+1}} (1-t^{m_i(\l)}) 
\end{equation*}
and
\begin{equation*}\label{eq:qbranch}
    \varphi_{\mu/\l} :=  \prod_{\substack{i > 0 \\  m_i(\mu) = m_i(\l)+1}} (1-t^{m_i(\mu)}) 
\end{equation*}
\end{definition}

The following branching rule is standard.

\begin{lemma} \label{thm:branching_formulas}
For $\l,\mu \in \Y$, we have
\begin{equation}\label{eq:skewP_branch_formula}
    P_{\l/\mu}(x_1,\ldots,x_k) = \sum_{\mu = \l^{(1)} \prec \l^{(2)} \prec \cdots \prec \l^{(k)}= \l} \prod_{i=1}^{k-1} x_i^{|\l^{(i+1)}|-|\l^{(i)}|}\psi_{\l^{(i+1)}/\l^{(i)}}
\end{equation}
and
\begin{equation}\label{eq:skewQ_branch_formula}
    Q_{\l/\mu}(x_1,\ldots,x_k) = \sum_{\mu = \l^{(1)} \prec \l^{(2)} \prec \cdots \prec \l^{(k)}=\l} \prod_{i=1}^{k-1} x_i^{|\l^{(i+1)}|-|\l^{(i)}|}\varphi_{\l^{(i+1)}/\l^{(i)}}.
\end{equation}
\end{lemma}

Hall-Littlewood polynomials satisfy the \emph{skew Cauchy identity}, upon which most probabilistic constructions rely.

\begin{proposition}\label{thm:finite_cauchy}
Let $\nu, \mu \in \Y$. Then
\begin{multline}\label{eq:finite_cauchy}
    \sum_{\kappa \in \Y} P_{\kappa/\nu}(x_1,\ldots,x_n;t)Q_{\kappa/\mu}(y_1,\ldots,y_m;t) \\
    = \prod_{\substack{1 \leq i \leq n \\ 1 \leq j \leq m}} \frac{1-tx_iy_j}{1-x_iy_j} \sum_{\l \in \Y} Q_{\nu/\l}(y_1,\ldots,y_m;t) P_{\mu/\l}(x_1,\ldots,x_n;t).
\end{multline}
\end{proposition}

For later convenience we set
\begin{equation}\label{eq:def_cauchy_kernel}
     \Pi(\bx;\by) := \prod_{\substack{1 \leq i \leq n \\ 1 \leq j \leq m}} \frac{1-tx_iy_j}{1-x_iy_j} = \exp\left(\sum_{\ell = 1}^\infty \frac{1-t^\ell}{\ell}p_\ell(\bx)p_\ell(\by)\right).
\end{equation}
The second equality in \eqref{eq:def_cauchy_kernel} is not immediate but is shown in \cite{mac}. The RHS of \eqref{eq:def_cauchy_kernel} makes sense as formal series in $\L \ot \L$, and \Cref{thm:finite_cauchy} generalizes straightforwardly with the skew $P$ and $Q$ functions in $x$ and $y$ replaced by corresponding elements of $\L \ot \L$.

To get a probability measure on $\cP$ from this identity, we would like homomorphisms $\phi: \L \to \C$ which take $P_\l$ and $Q_\l$ to $\R_{\geq 0}$. Such homomorphisms are called \emph{Hall-Littlewood nonnegative specializations} of $\L$. These were classified in \cite{matveev2019macdonald}, proving an earlier conjecture of \cite{gorin2014finite}: all are associated to triples $\{\alpha_n\}_{n \geq 1}, \{\beta_n\}_{n \geq 1}, \gamma$ such that $\gamma \geq 0$, $0 \leq \alpha_n,\beta_n < 1$ for all $ n \geq 1$, and $\sum_n \alpha_n, \sum_n \beta_n < \infty$. Given such a triple, the corresponding specialization is defined by 
\begin{align}\label{eq:p_specs}
\begin{split}
    p_1 &\mapsto \sum_{n \geq 1} \alpha_n + \frac{1}{1-t}\left(\gamma + \sum_{n \geq 1} \beta_n\right) \\
    p_k &\mapsto \sum_{n \geq 1} \alpha_n^k + \frac{(-1)^{k-1}}{1-t^k}\sum_{n \geq 1} \beta_n^k \quad \quad \text{ for all }k \geq 2.
\end{split}
\end{align}

We refer to a specialization as 
\begin{itemize}
    \item \emph{pure alpha} if $\gamma$ and all $\beta_n, n \geq 1$ are $0$.
    \item \emph{pure beta} if $\gamma$ and all $\alpha_n, n \geq 1$ are $0$.
    \item \emph{Plancherel} if all $\alpha_n,\beta_n, n \geq 1$ are $0$.
\end{itemize}

Note that the pure $\alpha$ specializations with finitely many nonzero $\alpha_i$ are given by simply plugging in $\alpha_i$ for the variable $x_i$ of the symmetric polynomial. 

\begin{remark}\label{rmk:spec_notation_and_plancherel}
Given a specialization $\psi$, we will write $P_\l(\psi)$ for $\psi(P_\l)$ and similarly for skew functions and $Q$ functions. In the case of a Plancherel specialization with parameter $\tau$, we write $P_\l(\gamma(\tau))$. In the case of a pure alpha specialization we write $P_\l(a_1,a_2,\ldots)$, which is natural since such a specialization may be seen as plugging in $a_i$'s for the variables as mentioned above. In the case where one variable $a$ is repeated $D$ times, we write $P_\lambda(a[D])$.
\end{remark}

Similarly to \eqref{eq:def_cauchy_kernel}, for any two specializations $\theta,\psi$ we let
\begin{equation}\label{eq:general_cauchy_kernel}
   \Pi(\theta;\psi) := \exp\left(\sum_{\ell = 1}^\infty \frac{1-t^\ell}{\ell}\theta(p_\ell)\psi(p_\ell)\right). 
\end{equation}
For any nonnegative specializations $\theta,\psi$ with 
\begin{equation}\label{eq:finiteness_cauchy}
    \sum_{\l \in \Y}P_\l(\theta)Q_\l(\psi) < \infty,
\end{equation}
the analogue 
\begin{equation}\label{eq:specialized_cauchy}
    \sum_{\kappa \in \Y} P_{\kappa/\nu}(\theta;t)Q_{\kappa/\mu}(\psi;t) \\
    =\Pi(\theta;\psi) \sum_{\l \in \Y} Q_{\nu/\l}(\psi;t) P_{\mu/\l}(\theta;t).
\end{equation}
of \eqref{eq:finite_cauchy} holds. 

The convergence of pure $\alpha$ specializations with small parameters to Plancherel specializations, shown in the following lemma, is standard. The lemma also shows this convergence to be monotonic on Hall-Littlewood polynomials for appropriate $t$, a useful fact for later convergence statements for which we are not aware of a reference.

\begin{lemma}
    \label{thm:alpha_gamma_monotonic}
    For any $\tau \geq 0, t \in [-1,1)$ and $\lambda,\mu \in \Y$, the sequence 
    \begin{equation} \label{eq:planch_sequence}
        P_{\l/\mu}\left(\tfrac{\tau}{1-t} \tfrac{1}{D}[D]\right), D = 1,2,\ldots
    \end{equation}
    is nondecreasing and converges to $P_{\l/\mu}(\gamma(\tau))$, and similarly for $Q$.
\end{lemma}
\begin{proof}
The fact that $\lim_{D \to \infty} P_{\l/\mu}\left(\tfrac{\tau}{1-t}\tfrac{1}{D}[D]\right) = P_{\l/\mu}(\gamma(\tau))$ is standard, and follows because (a) clearly $p_k(\frac{\tau}{1-t}\tfrac{1}{D}[D]) \to p_k(\gamma(\tau))$ for each $k$, and (b) $P_{\l/\mu}$ is a polynomial in the $p_k$.

Let us now prove the sequence \eqref{eq:planch_sequence} is nondecreasing. Specializing \Cref{thm:branching_formulas} to our case,
\begin{equation*}
    P_{\l/\mu}(\tfrac{\tau}{1-t} \tfrac{1}{D}[D]) = \sum_{\mu = \l^{(1)} \prec \ldots \prec \l^{(D)} = \l} \pfrac{\tau}{D(1-t)}^{|\l| - |\mu|} \prod_{i=1}^{D-1}\psi_{\l^{(i+1)}/\l^{(i)}}.
\end{equation*}
There are many distinct sequences $\mu = \l^{(1)} \prec \ldots \prec \l^{(D)} = \l$ for which the sets $\{\l^{(i)}: 1 \leq i \leq D\}$ are the same but the multiplicities of the partitions in the sequence are different, and we wish to group these together. Hence we collect terms according to the set of distinct partitions appearing:
\begin{equation}
    \label{eq:monotone_branching}
    P_{\l/\mu}(\tfrac{\tau}{1-t}\tfrac{1}{D}[D]) = \sum_{\substack{S \subset \Y }} \sum_{\substack{\mu = \l^{(1)} \prec \ldots \prec \l^{(D)} = \l \\ \{\l^{(1)},\ldots,\l^{(D)}\} = S}} \pfrac{\tau}{(1-t)D}^{|\l| - |\mu|}  \prod_{i=1}^{D-1}\psi_{\l^{(i+1)}/\l^{(i)}}.
\end{equation}
where clearly only sets $S$ of the form $\{\mu^{(1)},\ldots,\mu^{(k)}\}$ with $\mu = \mu^{(1)} \prec \ldots \prec \mu^{(k)} = \l$ contribute. We now fix such an $S$, and claim that the term 
\begin{equation}
    \label{eq:S_term}
    \sum_{\substack{\mu = \l^{(1)} \prec \ldots \prec \l^{(D)} = \l \\ \{\l^{(1)},\ldots,\l^{(D)}\} = S}} \pfrac{\tau}{(1-t)D}^{|\l| - |\mu|}  \prod_{i=1}^{D-1}\psi_{\l^{(i+1)}/\l^{(i)}}
\end{equation}
in \eqref{eq:monotone_branching} is nondecreasing in $D$. Note first that 
\[
\prod_{i=1}^{D-1}\psi_{\l^{(i+1)}/\l^{(i)}}
\]
is the same for all terms in \eqref{eq:S_term}, independent of $D$ (provided $D \geq |S|$ so the sum is non-empty), and that it is nonnegative because $t \in [-1,1]$. The number of summands in \eqref{eq:S_term} is $\binom{D-1}{|S|-1}$. Hence 
\begin{equation}
    \label{eq:S_term_simplified}
    \text{\eqref{eq:S_term}} = \bbone(D \geq |S|) \prod_{i=1}^{|S|} \psi_{\mu^{(i+1)}/\mu^{(i)}} \pfrac{\tau}{(1-t)D}^{|\l| - |\mu|} \binom{D-1}{|S|-1}.
\end{equation}
Since the RHS of \eqref{eq:S_term_simplified} is nonnegative, we need only show it is nondecreasing when $D \geq |S|$, as otherwise it is $0$. The ratio of successive (nonzero) terms is
\begin{align*}
  \frac{\pfrac{\tau}{D+1}^{|\l| - |\mu|} \binom{D}{|S|-1}}{\pfrac{\tau}{(1-t)D}^{|\l| - |\mu|} \binom{D-1}{|S|-1}} &= \pfrac{D}{D+1}^{|\l| - |\mu|} \frac{D}{D-|S|+1} \\
  &\geq \left(1-\frac{1}{D+1}\right)^{|\l| - |\mu|} \frac{D}{D-(|\l|-|\mu|+1)+1} \\
  & \geq \left(1-\frac{|\l| - |\mu|}{D+1}\right) \frac{D}{D-(|\l|-|\mu|)} \\
  & \geq 1.
\end{align*}
In the first inequality we used that $|S| \leq |\l|+1$, as the sizes of the partitions in $S$ must each differ by at least one. In the second we used the elementary inequality $(1-x)^n \geq 1-nx$ for $x \in [0,1], n \geq 1$, which follows by noting equality holds at $x=0$ and the LHS has larger derivative on the interval. This completes the proof.
\end{proof}

We comment that the above result and proof also holds for Macdonald polynomials with $q,t \in (-1,1)$. A final useful fact for us is a simple explicit formula for the Hall-Littlewood polynomials when a geometric progression $u,ut,\ldots,ut^{n-1}$ is substituted for $x_1,\ldots,x_n$---this is often referred as a \emph{principal specialization}. Let 
\begin{equation}
    n(\l) := \sum_{i=1}^n (i-1)\l_i = \sum_{x \geq 1} \binom{\l_x'}{2}.
\end{equation}
The following formula is standard and may be easily derived from \eqref{eq:hlP_formula}, and also follows directly from \cite[Ch. III.2, Ex. 1]{mac}.

\begin{proposition}[Principal specialization formula]\label{thm:hl_principal_formulas}
For $\l \in \Y_n$,
\begin{align*}
    P_\l(u,ut,\ldots,ut^{n-1};t) &= u^{|\l|} t^{n(\l)} \frac{(t;t)_n}{(t;t)_{n-\len(\l)}\prod_{i \geq 1} (t;t)_{m_i(\l)}}.
\end{align*}
\end{proposition}

\subsection{Markov dynamics on partitions and Hall-Littlewood processes.} 

One obtains probability measures on sequences of partitions using \eqref{eq:specialized_cauchy} as follows.

\begin{definition}\label{def:HL_proc}
Let $\theta$ and $\psi_1,\ldots,\psi_k$ be Hall-Littlewood nonnegative specializations such that each pair $\theta,\psi_i$ satisfies \eqref{eq:finiteness_cauchy}. The associated \emph{ascending Hall-Littlewood process} is the probability measure on sequences $\l^{(1)},\ldots,\l^{(k)}$ given by 
\[
\Pr(\l^{(1)},\ldots,\l^{(k)}) = \frac{Q_{\l^{(1)}}(\psi_1) Q_{\l^{(2)}/\l^{(1)}}(\psi_2) \cdots Q_{\l^{(k)}/\l^{(k-1)}}(\psi_k) P_{\l^{(k)}}(\theta)}{\prod_{i=1}^k \Pi(\psi_i;\theta)}.
\]
\end{definition}

The $k=1$ case of \Cref{def:HL_proc} is a measure on partitions, referred to as a \emph{Hall-Littlewood measure}. Instead of defining joint distributions all at once as above, one can define Markov transition kernels on $\cP$. 

\begin{definition}\label{def:cauchy_dynamics}
Let $\theta,\psi$ be Hall-Littlewood nonnegative specializations satisfying \eqref{eq:finiteness_cauchy} and $\l$ be such that $P_\l(\theta) \neq 0$. The associated \emph{Cauchy Markov kernel} is defined by 
\begin{equation}\label{eq:def_HL_cauchy_dynamics}
    \Pr(\l \to \nu) = Q_{\nu/\l}(\psi) \frac{P_\nu(\theta)}{P_\l(\theta) \Pi(\psi; \theta)}.
\end{equation}
\end{definition}

It is clear that the ascending Hall-Littlewood process above is nothing more than the joint distribution of $k$ steps of a Cauchy Markov kernel with specializations $\psi_i,\theta$ at the $i\tth$ step.

\section{Between the slowed $t$-TASEP and Hall-Littlewood processes}\label{sec:sampling}

In this section we formally define slowed $t$-TASEP, and show in \Cref{thm:HL_qTASEP_connection} that it is equivalent (in the case of packed initial condition) to a Hall-Littlewood process with one Plancherel specialization and one principal specialization $1,t,\ldots$.

\begin{definition}
Let 
\[\mathbb{X} := \{(x_1,x_2,\ldots) \in \Z^\N: x_1 > x_2 > \cdots\}\]
be the space of particle configurations on $\Z$, where the $x_i$ is the position of the $i\tth$ particle from the right, and 
\[
\mathbb{X}_0 := \{(x_1,x_2,\ldots) \in \mathbb{X}: x_i = -i \text{ for all sufficiently large }i\}.
\]
We denote particle configurations $(x_1,x_2,\ldots)$ by $\bx$, and if $x_{k-1} > x_k+1$ we write $\bx^k := (x_1,\ldots,x_{k-1},x_k+1,x_{k+1},\ldots)$.
\end{definition}

\begin{definition}\label{def:slow_q-TASEP}
\emph{Slowed $t$-TASEP with initial condition $\bx_0 \in \mathbb{X}$} is the continuous-time stochastic process $\bx_t(\tau) = (x_1(\tau),x_2(\tau),\ldots)$ on $\mathbb{X}$ in which $\bx_t(0)=\bx_0$ and the particles at positions $x_k,k \geq 1$ each have independent Poisson clocks with rates $t^{x_k+k}(1-t^{x_{k-1}-x_k-1})$, and jump to the right by $1$ when they ring. Equivalently, it is defined by the matrix of transition rates
\begin{equation}\label{eq:qTASEP_trans_rates}
    \dde{}{\tau}{\tau=0} \Pr(\bx_t(T+\tau) = \by| \bx_t(T) = \bx) = 
    \begin{cases}
    t^{x_k+k}(1-t^{x_{k-1}-x_k-1}) & \by = \bx^k \text{ for some $k \in \Z_{>0}$} \\
    -1 & \by = \bx \\
    0 & \text{ otherwise}
    \end{cases}.
\end{equation}
We refer to the initial condition $(-1,-2,\ldots)$ as \emph{packed}.
\end{definition}

Recall the notation of Plancherel and infinite alpha specializations from \Cref{rmk:spec_notation_and_plancherel}.

\begin{theorem}\label{thm:HL_qTASEP_connection}
Let $\l(\tau), \tau \geq 0$ be the stochastic process on $\cP$ distributed as an ascending Hall-Littlewood process with one specialization $1,t,\ldots$ and the other Plancherel, i.e. with finite-dimensional marginals
\[
\Pr(\l(\tau_i) = \l^{(i)} \text{ for all }i=1,\ldots,k) =  \frac{\left(\prod_{j=1}^k Q_{\l^{(j)}/\l^{(j-1)}}(\gamma(\tau_j-\tau_{j-1});t) \right) P_{\l^{(k)}}(1,t,\ldots;t)}{\exp\left(\frac{\tau_k}{1-t}\right)}.
\]
Then
\[\bx_t(\tau) = \left(\l_k'((1-t)\tau) - k\right)_{k \geq 1} \]
in (multi-time) distribution, where $\bx_t(\tau)$ is a slowed $t$-TASEP with packed initial condition and parameter $t$.
\end{theorem}

The proof consists of computing the explicit transition dynamics of $\l(\tau)$ and verifying that they correspond to those of $\bx_t(\tau)$, which are already given explicitly. This is done in the following proposition.

\begin{proposition}\label{thm:explicit_planch_dynamics}
Let $n \in \N \cup \{\infty\}$ and $\l(\tau), \tau \geq 0$ be the stochastic process on $\cP_n$ in continuous time $\tau$ with finite-dimensional marginals given by the Hall-Littlewood process
\[
\Pr(\l(\tau_i) = \l^{(i)} \text{ for all }i=1,\ldots,k) =  \frac{\left(\prod_{j=1}^k Q_{\l^{(j)}/\l^{(j-1)}}(\gamma(\tau_j-\tau_{j-1});t) \right) P_{\l^{(k)}}(1,t,\ldots,t^{n-1};t)}{\exp\left(\frac{\tau_k(1-t^n)}{1-t}\right)}
\]
for each sequence of times $0 \leq \tau_1 \leq \tau_2 \leq \cdots \leq \tau_k$ and $\l^{(1)},\ldots,\l^{(k)} \in \cP_n$, where in the product we take the convention $\tau_0=0$ and $\l^{(0)}$ is the zero partition. Then $\l(\tau)$ has the following explicit description: For each $i$, $\l_i(\tau)$ has an exponential clock of rate $(1-t^n)t^{i-1}$, where if $n=\infty$ we take $t^n=0$. When $\l_i$'s clock rings, $\l_i(\tau)$ increases by $1$, and if this violates the weakly decreasing condition in $(\l_1(\tau),\ldots,\l_n(\tau))$, then the tuple is reordered to again be weakly decreasing. 
\end{proposition}

\begin{remark}
These dynamics can be equivalently interpreted as a Poisson random walk in $\Z^n$ with rates $1,t,\ldots,t^{n-1}$ reflected at the boundary of the positive type $A$ Weyl chamber. We note also that they are continuous-time limits of the $\alpha$ dynamics described previously in \cite[Proposition 5.3]{van2020limits}.
\end{remark}

\begin{proof}[Proof of {\Cref{thm:explicit_planch_dynamics}}]
Define the $\cP_n \times \cP_n$ matrix $B$, which goes by the name of the generator, $Q$-matrix\footnote{Usually the matrix $B$ would also be called $Q$, but we have chosen $B$ to avoid confusion with the polynomials $Q_{\nu/\l}$.}, or matrix of transition rates of the Markov process $\l(\tau)$, by
\begin{equation}
    B(\mu,\nu) := \dde{}{\tau}{\tau=0} \Pr(\lambda(T+\tau)=\nu | \lambda(T) = \mu)
\end{equation}
(this is of course independent of $T$ by the Markov property). By the equivalence of the Hall-Littlewood process with the Cauchy dynamics of \Cref{def:cauchy_dynamics} and the explicit formulas \eqref{eq:def_cauchy_kernel}, we have 
\[
\Pr(\lambda(T+\tau)=\nu | \lambda(T) = \mu) = \frac{Q_{\nu/\mu}(\gamma(\tau);t)}{\exp\left(\tau \frac{1-t^n}{1-t}\right)} \frac{P_\nu(1,\ldots,t^{n-1};t)}{P_\mu(1,\ldots,t^{n-1};t)},
\]
and only the term 
\[
\frac{Q_{\nu/\mu}(\gamma(\tau);t)}{\exp\left(\tau \frac{1-t^n}{1-t}\right)}
\]
depends on $\tau$. When $\nu = \mu$, $Q_{\nu/\mu} = 1$, so 
\begin{equation}\label{eq:gen_diagonal}
    B(\mu,\mu) = \dde{}{\tau}{\tau=0} \Pr(\lambda(T+\tau)=\mu | \lambda(T) = \mu) = -\frac{1-t^n}{1-t}.
\end{equation}
In general, $Q_{\nu/\mu}$ (viewed as an element of the ring of symmetric functions) is a polynomial in the $p_k, k \geq 1$ which is homogeneous of degree $|\nu|-|\mu|$ if we define each $p_k$ to have degree $k$. Under the Plancherel specialization, all $p_k, k \geq 2$ are sent to $0$, hence $Q_{\nu/\mu}(\gamma(\tau);t) = O(\tau^{|\nu|-|\mu|})$ as $\tau \to 0$. It follows that 
\begin{equation}\label{eq:zero_condition}
    B(\mu,\nu) = 0 \text{ if $|\nu| > |\mu|+1$}.
\end{equation}
When $|\nu|=|\mu|+1$, it follows from \Cref{thm:branching_formulas} that $Q_{\nu/\mu} = \varphi_{\nu/\mu} p_1$. Hence 
\[
Q_{\nu/\mu}(\gamma(\tau);t) = \varphi_{\nu/\mu}\frac{\tau}{1-t}
\]
and 
\begin{equation}\label{eq:deriv_1box}
    \dde{}{\tau}{\tau=0} \Pr(\lambda(T+\tau)=\nu | \lambda(T) = \mu) = \frac{\varphi_{\nu/\mu}}{1-t} \frac{P_\nu(1,\ldots,t^{n-1};t)}{P_\mu(1,\ldots,t^{n-1};t)}, 
\end{equation}
Since $\mu \prec \nu$, all parts in $\mu$ and $\nu$ are the same except for one part which differs by $1$, so there are some integers $k,a \geq 0, b \geq 1$ such that $\mu = (\ldots,(k+1)[a],k[b],\ldots)$ and $\nu = (\ldots,(k+1)[a+1],k[b-1],\ldots)$ where we use $x[c]$ to denote $x$ repeated $c$ times in the partition. Let $\ell$ be the smallest integer so that $\mu_\ell = k$. To compute \eqref{eq:deriv_1box} we specialize \Cref{def:psi_varphi_coefs} to obtain
\[
\varphi_{\nu/\mu} = 1-t^{a+1},
\]
and plugging this and \Cref{thm:hl_principal_formulas} into \eqref{eq:deriv_1box} yields that in our situation
\begin{equation}\label{eq:deriv_1box_simpler}
    B(\mu,\nu) = \frac{t^{\ell-1}(1-t^b)}{1-t} = t^{\ell-1} + \ldots + t^{(\ell+b-1)-1}.
\end{equation}
Combining \eqref{eq:gen_diagonal}, \eqref{eq:zero_condition} and \eqref{eq:deriv_1box_simpler} yields a complete description of $B$. As a sanity check, these explicit formulas tell us that for any fixed $\mu$, 
\[
\sum_{\nu \in \cP} B(\mu,\nu) = 0,
\]
which should be true in general because $B$ comes from a stochastic matrix. We next note that this is exactly the $Q$-matrix of the Poisson jump process described in the statement, since if the $\ell\tth$ part of the partition has rate $t^{\ell-1}$, then the rate at which any one of the parts equal to $k$ jump in the above setup is 
\[
t^{\ell-1} + \ldots + t^{(\ell+b-1)-1} = t^{\ell-1}\frac{1-t^b}{1-t}.
\]
It follows by results of \cite{feller2015integro}, see also \cite[Sec. 4.1]{borodin2012markov}, that a continuous-time Markov process on a countable state space is uniquely determined by its $Q$-matrix in the case when the diagonal entries of $Q$ (hence all entries, by stochasticity) are uniformly bounded; in this case, the time-$\tau$ Markov evolution operator is given by $e^{tB}$. The diagonal entries of our matrix $B$ are all the same, in particular uniformly bounded, hence both of our Markov processes are uniquely determined by their (equal) $Q$-matrices and the result follows.
\end{proof}

\begin{proof}[Proof of {\Cref{thm:HL_qTASEP_connection}}]
It suffices to observe that the transition rates \eqref{eq:qTASEP_trans_rates} match those in the proof of \Cref{thm:explicit_planch_dynamics} after translating between $\cP$ and $\mathbb{X}_0$ as above.
\end{proof}

\section{A contour integral formula for {$t$}-moments} \label{sec:q_moments} 

In this section we we prove contour integral formulas for certain $t$-moment observables of this particle system, which will be the main tool in subsequent asymptotic results. We again take $t \in (0,1)$, and denote the Weyl denominator/Vandermonde determinant by 
\[
\Delta(z_1,\ldots,z_n) := \prod_{1 \leq i < j \leq n} (z_i - z_j).
\]

\begin{proposition}\label{thm:observable_formula}
Let $\l(\tau)$ be distributed as a Hall-Littlewood process with specializations $1,t,\ldots$ and $\gamma(\tau)$ as in \Cref{thm:explicit_planch_dynamics}. Then for any positive integers $r_1,\ldots,r_M$,
\begin{multline}\label{eq:t-moment_integral}
    \E\left[t^{-\sum_{m=1}^M \sum_{j=1}^{r_m}\l_j'(\tau)}\right]  = \prod_{m=1}^M \frac{(-1)^{\binom{r_m}{2}}}{r_m!(2 \pi i)^{r_m}} \displaystyle \oint \cdots \oint \prod_{m=1}^M \left( \Delta(z_{1,m},\ldots,z_{r_m,m})^2 \prod_{s=1}^{r_m} \frac{e^{\tau z_{s,m}} (1+t^{-1}z_{s,m}^{-1})}{z_{s,m}^{r_m}}  \right)\\ \cdot \prod_{1 \leq \alpha < \beta \leq M } \prod_{\substack{ 1 \leq i \leq r_\alpha \\ 1 \leq j \leq r_\beta}} \frac{1-z_{j,\beta}/z_{i,\alpha}}{1-t^{-1}z_{j,\beta}/z_{i,\alpha}} dz_{1,1} \cdots dz_{r_M,M}.
\end{multline}
with all contours encircling $0$ and satisfying
\begin{align*}
    |z_{j,\beta}| < t|z_{i,\alpha}| & \text{ for all }1 \leq \alpha < \beta \leq M, 1 \leq i \leq r_\alpha, 1 \leq j \leq r_\beta \\
    |z_{s,\alpha}| < t^{-1} & \text{ for all }1 \leq \alpha \leq M, 1 \leq s \leq r_\alpha.
\end{align*}
\end{proposition}
\begin{proof}
First consider a partition $\mu(D,\tau)$ distributed as a Hall-Littlewood process with alpha specializations $1,t,\ldots$ and $\left(\frac{\tau}{1-t} \frac{1}{D}\right)[D]$, where again $[D]$ denotes $D$ copies of the same specialization. We recall that this latter specialization is an approximation to the Plancherel specialization $\gamma(\tau)$ and converges to it as $D \to \infty$ by \Cref{thm:alpha_gamma_monotonic}. Specializing \cite[Thm. 2.12]{bufetov2018hall} to our case\footnote{In the notation of \cite[Thm. 2.12]{bufetov2018hall}, we are taking $N=M$, $X_1 = \left(\frac{\tau}{1-t} \frac{1}{D}\right)[D], X_2 = \cdots = X_M = 0$ and $Y_1 = \cdots = Y_{M-1} = 0, Y_M =  1,t,\ldots$. Then the product over $1 \leq \alpha \leq \beta \leq M$ in the fourth line of (2.22) of \cite{bufetov2018hall} only gives nontrivial terms when $\alpha = 1$ or $\beta = M$.} yields 
\begin{align}\label{eq:bufmat_formula}
\begin{split}
&\E\left[t^{-\sum_{m=1}^M \sum_{j=1}^{r_m}\mu(D,\tau)_j'}\right]  = \prod_{m=1}^M \frac{(-1)^{\binom{r_m}{2}}}{r_m!(2 \pi i)^{r_m}} \displaystyle \oint \cdots \oint \\
&\prod_{m=1}^M \left( \frac{\prod_{1 \leq i < j \leq r_m} (z_{i,m} - z_{j,m})^2 }{(z_{1,m}\cdots z_{r_m,m})^{r_m}} \prod_{s=1}^{r_m} \left(\prod_{j =1}^D \frac{1+z_{s,m}\frac{\tau}{1-t} \frac{1}{D}}{1 + tz_{s,m} \frac{\tau}{1-t} \frac{1}{D}} \right)  \left(\prod_{i \geq 1} \frac{1 + z_{s,m}^{-1}t^{i-2}}{1 + z_{s,m}^{-1}t^{i-1}}\right) \right)
\\ &\cdot \prod_{1 \leq \alpha < \beta \leq M } \prod_{\substack{ 1 \leq i \leq r_\alpha \\ 1 \leq j \leq r_\beta}} \frac{1-z_{j,\beta}/z_{i,\alpha}}{1-t^{-1}z_{j,\beta}/z_{i,\alpha}} dz_{1,1} \cdots dz_{r_M,M}
\end{split}
\end{align}
with all contours encircling $0$ and 
\begin{align*}
    |z_{i,\alpha}| < t|z_{j,\beta}|  & \quad \quad \text{ for all } \quad \quad 1 \leq \alpha < \beta \leq M, 1 \leq i \leq r_\alpha, 1 \leq j \leq r_\beta \\
    \frac{\tau}{1-t} \frac{1}{D} < |z_{s,\alpha}| < t^{-1} & \quad \quad \text{ for all } \quad \quad 1 \leq \alpha \leq M, 1 \leq s \leq r_\alpha 
\end{align*}
provided such contours exist. 
We note that for any fixed $t \in (0,1)$ and $\tau \geq 0$, such contours exist for all $D$ sufficiently large. Picking a choice of contours, we have 
\begin{align}\label{eq:bufmat_rhs_limit}
\begin{split}
    &\lim_{D \to \infty} \text{RHS}{\eqref{eq:bufmat_formula}} = \prod_{m=1}^M \frac{(-1)^{\binom{r_m}{2}}}{r_m!(2 \pi i)^{r_m}} \displaystyle \oint \cdots \oint \prod_{m=1}^M \left( \Delta(z_{1,m},\ldots,z_{r_m,m})^2 \prod_{s=1}^{r_m} \frac{(1+t^{-1}z_{s,m}^{-1}) e^{\tau z_{s,m}}}{z_{s,m}^{r_m}} \right)
\\ &\cdot \prod_{1 \leq \alpha < \beta \leq M } \prod_{\substack{ 1 \leq i \leq r_\alpha \\ 1 \leq j \leq r_\beta}} \frac{1-z_{j,\beta}/z_{i,\alpha}}{1-t^{-1}z_{j,\beta}/z_{i,\alpha}} dz_{1,1} \cdots dz_{r_M,M}
    \end{split}
\end{align}
where the limit commutes with the integral because the integrand remains bounded as $D \to \infty$ and the contours are compact. 

It now suffices to show convergence of the left hand side of \eqref{eq:bufmat_formula} to that of \eqref{eq:t-moment_integral}, i.e. we must show
\begin{align}
    \label{eq:bufmat_lhs_limit}
    \begin{split}
        &\lim_{D \to \infty} \frac{1}{\Pi\left(1,t,\ldots;\frac{\tau}{1-t} \frac{1}{D}[D]\right)} \sum_{\l \in \Y} Q_\l(1,t,\ldots) P_\l\left( \frac{\tau}{1-t} \frac{1}{D}[D]\right) t^{-\sum_{m=1}^M \sum_{j=1}^{r_m}\l_j'} \\
        &= \frac{1}{\Pi(1,t,\ldots;\gamma(\tau))} \sum_{\l \in \Y} Q_\l(1,t,\ldots)P_\l(\gamma(\tau)) t^{-\sum_{m=1}^M \sum_{j=1}^{r_m}\l_j'}
    \end{split}
\end{align}

It follows simply from the definition by \eqref{eq:general_cauchy_kernel} that 
\[
\lim_{D \to \infty} \frac{1}{\Pi\left(1,t,\ldots;\frac{\tau}{1-t} \frac{1}{D}[D]\right)} = \frac{1}{\Pi(1,t,\ldots;\gamma(\tau))},
\]
so it suffices to show that the sum on the LHS of \eqref{eq:bufmat_lhs_limit} converges to the one on the RHS. We may write the sum on the LHS (resp. RHS) as an integral of the function $f_D(\l) = Q_\l(1,t,\ldots) P_\l(\frac{\tau}{1-t}\frac{1}{D}[D])$ (resp. $f(\l) = Q_\l(1,t,\ldots) P_\l(\gamma(\tau))$) with respect to the measure on the discrete set $\Y$ determined by $\text{meas}(\{\l\}) = t^{-\sum_{m=1}^M \sum_{j=1}^{r_m}\l_j'}$. By \Cref{thm:alpha_gamma_monotonic}, $f_D(\l)$ converges monotonically from below to $f(\l)$, hence the monotone convergence theorem yields the desired convergence of sums, and \eqref{eq:bufmat_lhs_limit} follows.

The limit $D \to \infty$ also makes the contour condition $\frac{\tau}{1-t} \frac{1}{D} < |z_{s,\alpha}|$ automatic, so combining \eqref{eq:bufmat_rhs_limit} and \eqref{eq:bufmat_lhs_limit} completes the proof.
\end{proof}

\section{Moment convergence at fixed $t$}\label{sec:fixed_t}

In this section, as before, we take $t \in (0,1)$. While in future sections we will consider a limit as $t \to 1$, in this one we prove that joint moments of the exponential transforms $t^{-(\l_i'(\tau) + \log_t \tau)}$ converge. This supports the claim in the Introduction that the $\l_i'(\tau)$---hence the particles $x_i(\tau)$ in slowed $t$-TASEP---have asymptotically finite fluctuations about a mean $-\log_t \tau + O(1)$ as $\tau \to \infty$ for fixed $t$.

\begin{proposition}\label{thm:limit_moments}
For any $M \in \Z_{\geq 1}$ and $r_1,\ldots,r_M \in \Z_{\geq 1}$,
\begin{align}
\begin{split}
    &\lim_{\tau \to \infty} \E\left[t^{-\sum_{m=1}^M \sum_{j=1}^{r_m} (\l_i'(\tau) + \log_t \tau)}\right] = \prod_{m=1}^M \frac{(-1)^{\binom{r_m}{2}}}{r_m!t^{r_m}(2 \pi i)^{r_m}}  \displaystyle \oint \cdots \oint \\
    &\prod_{m=1}^M \left( \Delta(w_{1,m},\ldots,w_{r_m,m})^2 \prod_{s=1}^{r_m} \frac{e^{w_{s,m}}}{w_{s,m}^{r_m+1}}  \right) \prod_{1 \leq \alpha < \beta \leq M } \prod_{\substack{ 1 \leq i \leq r_\alpha \\ 1 \leq j \leq r_\beta}} \frac{1-w_{j,\beta}/w_{i,\alpha}}{1-t^{-1}w_{j,\beta}/w_{i,\alpha}} dw_{1,1} \cdots dw_{r_M,M}
\end{split}
\end{align}
with all contours encircling $0$ and satisfying
\begin{align*}
    |w_{j,\beta}| < t|w_{i,\alpha}| & \text{ for all }1 \leq \alpha < \beta \leq M, 1 \leq i \leq r_\alpha, 1 \leq j \leq r_\beta.
\end{align*}
\end{proposition}
\begin{proof}
Making the change of variables $w_{s,m} = \tau z_{s,m}$, and choosing the $z$-contours of \Cref{thm:observable_formula} to scale as $\tau^{-1}$ such that the $w$-contours are independent of $\tau$, \Cref{thm:observable_formula} yields
\begin{multline*}
\E\left[t^{-\sum_{m=1}^M \sum_{j=1}^{r_m} (\l_j'(\tau)+\log_t \tau)}\right] = \tau^{-\sum_{m=1}^M r_m} \E\left[t^{-\sum_{m=1}^M \sum_{j=1}^{r_m}  \l_j'(\tau)}\right] \\
=\prod_{m=1}^M \frac{(-1)^{\binom{r_m}{2}}}{r_m!(2 \pi i)^{r_m}} \displaystyle \oint \cdots \oint \prod_{m=1}^M \left( \Delta(w_{1,m},\ldots,w_{r_m,m})^2 \prod_{s=1}^{r_m} \frac{e^{w_{s,m}}(1+t^{-1}\tau w_{s,m}^{-1})/\tau}{w_{s,m}^{r_m}}  \right)\\ \cdot \prod_{1 \leq \alpha < \beta \leq M } \prod_{\substack{ 1 \leq i \leq r_\alpha \\ 1 \leq j \leq r_\beta}} \frac{1-w_{j,\beta}/w_{i,\alpha}}{1-t^{-1}w_{j,\beta}/w_{i,\alpha}} dw_{1,1} \cdots dw_{r_M,M}
\end{multline*}
with contours as in the statement above. The integrand is clearly bounded independent of $\tau$, so by bounded convergence the $\tau \to \infty$ limit passes through the integral.
\end{proof}

In particular, the random variables $t^{-\sum_{i=1}^r (\l_i'(\tau) + \log_t \tau) }$ have asymptotically bounded variance. We note, however, that the moment convergence above does \emph{not} imply that these random variables converge in distribution. For instance, suppose $\tau \to \infty$ along the subsequence $\tau^{(\alpha)}_n = t^{-n-\alpha}$ for some $\alpha \in [0,1)$. Then $t^{-(\l_1'(\tau^{(\alpha)}_n)+\log_t \tau^{(\alpha)}_n)}$ is supported on $t^{\alpha+\Z}$ for all $n$, since $\l_1'$ takes values in $\Z$. It is therefore clear that for different $\alpha$, these random variables cannot converge to the same limit (except in the case of a trivial limit such as $0$, which may be ruled out by \Cref{thm:limit_moments}). Numerically, the moments in \Cref{thm:limit_moments} may be checked to violate Carleman's condition, so these nonunique limits are not surprising. We conjecture that random variables $\l_i'(t^{-n-\alpha}) -(n+\alpha)$ converge in joint distribution as $n \to \infty$ for any $\alpha \in [0,1)$, with distinct limits for each $\alpha$, which nonetheless have the same joint moments. However, since our main purpose is the joint $t \to 1, \tau \to \infty$ limit, we do not address this question.

\section{Law of large numbers}\label{sec:particle_LLN}

In this section we establish the law of large numbers for particle positions as $t \to 1$, recalled below.

\llnintro*

We begin with a straightforward heuristic derivation by taking a continuum limit of jump rates to obtain an ODE for particle positions, then give a rigorous proof using the observables in \Cref{thm:observable_formula}.

First, we wish to see the scaling of $t=e^{-\eps}$, space and time such that both the particle positions and jump rates converge to nontrivial limits. The first particle $x_1(\tau)$ jumps as a rate-$1$ Poisson process as $t \to 1$, so for it to converge to a nontrivial limit, we should wait a long time (take time to be $\approx \tau/\eps$) and rescale space by $\eps$, i.e. we should consider $\eps x_1\left( \tau/\eps \right)$. For the first particle, the law of large numbers guarantees concentration, though arguing for the others would be slightly more involved; however, let us suppose 
\[
\eps x_k(\tau/\eps) \to c_k(\tau)
\]
for some functions $c_1,c_2,\ldots$.

Then we have convergence of jump rates
\[
t^{x_k(\tau/\eps)+k}(1-t^{x_{k-1}(\tau/\eps)-x_k(\tau/\eps)-1}) \to e^{-c_k(\tau)}(1-e^{c_{k-1}(\tau) - c_k(\tau)}) = e^{-c_k(\tau)} - e^{-c_{k-1}(\tau)}
\]
where when $k=1$ we take $e^{-c_{k-1}(\tau)} = 0$.
Because we are scaling time as $\eps^{-1}$ and then rescaling space by $\eps$, by concentration of Poisson variables we should have
\[
\lim_{\eps \to 0^+} \left(\text{jump rate of }x_k(\tau/\eps)\right) = \dderiv{c_k(\tau)}{\tau}.
\]
Hence the functions $c_k(\tau)$ should satisfy
\begin{align}\label{eq:c_ODE}
    \dderiv{c_k(\tau)}{\tau} &= e^{-c_k(\tau)} - e^{-c_{k-1}(\tau)} \quad \quad \quad \text{ for all }k \geq 1,
\end{align}
where when $k=1$ we take $c_0 = \infty$ so the second term on the RHS is not present.
It is easy to verify by inspection that the limits given in \Cref{thm:LLN_intro}, namely
\begin{equation}\label{eq:def_c}
    c_k(\tau) := \log \left(\sum_{j=0}^k \frac{\tau^j}{j!}\right) - \log \left(\sum_{j=0}^{k-1} \frac{\tau^j}{j!} \right),
\end{equation}
satisfy \eqref{eq:c_ODE}, and furthermore have initial conditions $c_k(0) = 0$ as they should. This concludes the heuristic derivation of \Cref{thm:LLN_intro}, and we move on to the proof.

\begin{proof}[Proof of {\Cref{thm:LLN_intro}}]
We first claim that it suffices to show the same limit as in \Cref{thm:LLN_intro} for the slightly different quantity $\eps x_k\pfrac{\tau}{1-e^{-\eps}}$. Assuming this result and setting $\teps = -\log(1-\eps)$ so that $\eps = 1-e^{-\teps}$, we have
\[
\eps x_k(\tau/\eps) = \left(1-e^{-\teps}\right) x_k \pfrac{\tau}{1-e^{-\teps}} = \left(\teps + O(\teps^2)\right) x_k\pfrac{\tau}{1-e^{-\teps}},
\]
hence convergence in probability for $\eps x_k\pfrac{\tau}{1-e^{-\eps}}$ implies the same for $\eps x_k(\tau/\eps)$. Since 
\[
x_k\pfrac{\tau}{1-t} = \l_k(\tau)' - k \quad \quad \text{in distribution}
\]
by \Cref{thm:HL_qTASEP_connection}, it suffices to show that 
\begin{equation}\label{eq:prob_conv_lambda}
    \eps \l_k(\tau) \to \log \left(\sum_{j=0}^k \frac{\tau^j}{j!}\right) - \log \left(\sum_{j=0}^{k-1} \frac{\tau^j}{j!} \right) \quad \quad \text{in probability as $\eps \to 0^+$.}
\end{equation}
It therefore suffices to show the convergence of Laplace transforms
\begin{equation}\label{eq:prob_conv_laplace}
    e^{\eps \sum_{j=1}^r \l_j(\tau)} \to \sum_{j=0}^r \frac{\tau^j}{j!} \quad \quad \text{in probability as $\eps \to 0^+$}
\end{equation}
for each $r$, as then $\eps \sum_{j=1}^r \l_j(\tau)$ converges in probability, and taking differences yields \eqref{eq:prob_conv_lambda}. Let $Y_r(t) = e^{\eps \sum_{j=1}^r \l_j(\tau)}$ (recall $t=e^{-\eps}$). By Chebyshev's inequality, to show \eqref{eq:prob_conv_laplace} it suffices to show 
\begin{equation}\label{eq:exp_conv}
    \E[Y_r(t)] \to  \sum_{j=0}^r \frac{\tau^j}{j!} 
\end{equation}
and 
\begin{equation}\label{eq:var_conv}
    \var(Y_r(t)) \to 0
\end{equation}
as $\eps \to 0^+$.

By \Cref{thm:observable_formula}, 
\begin{equation}\label{eq:exp_first}
   \E[Y_r(t)] = \frac{(-1)^{\binom{r}{2}}}{r! (2 \pi \bi)^r} \oint \cdots \oint \prod_{1 \leq i < j \leq r} (z_i - z_j)^2 \prod_{s=1}^r \frac{1+t^{-1}z_s^{-1}}{z_s^{r-1}}e^{\tau z_s} \frac{dz_s}{z_s},
\end{equation}
where all contours are circles around the origin of radius $\leq 1$, and
\begin{align}\label{eq:exp_squared}
\begin{split}
    \E[Y_r(t)^2] =& \frac{1}{(r!)^2 (2 \pi \bi)^{2r}} \oint \cdots \oint \prod_{1 \leq i,j \leq r} \frac{1 - z_{j,2}/z_{i,1}}{1-t^{-1}z_{j,2}/z_{i,1}} \\
    &\cdot \prod_{\ell=1}^2\left( \prod_{1 \leq i < j \leq r} (z_{i,\ell}-z_{j,\ell})^2 \prod_{s=1}^r \frac{1+t^{-1}z_{s,\ell}^{-1}}{z_{s,\ell}^{r-1}}e^{\tau z_{s,\ell}} \frac{d z_{s,\ell}}{z_{s,\ell}}\right)
\end{split}
\end{align}
where we take the $z_{s,1}$ and $z_{s,2}$ contours to be circles of radii $1$ and $R$ respectively, for some $0 < R < 1$ fixed independent of $\eps$ (for all $\eps$ sufficiently small that such contours satisfy the conditions in \Cref{thm:observable_formula}). It follows by combining two instances of \eqref{eq:exp_first} that
\begin{equation*}
    \E[Y_r(t)]^2 = \frac{1}{(r!)^2 (2 \pi \bi)^{2r}} \oint \cdots \oint  \prod_{\ell=1}^2\left( \prod_{1 \leq i < j \leq r} (z_{i,\ell}-z_{j,\ell})^2 \prod_{s=1}^r \frac{1+t^{-1}z_{s,\ell}^{-1}}{z_{s,\ell}^{r-1}}e^{\tau z_{s,\ell}} \frac{d z_{s,\ell}}{z_{s,\ell}}\right)
\end{equation*}
with the same contours as in \eqref{eq:exp_squared}, and combining with \eqref{eq:exp_squared} yields
\begin{align}\label{eq:var_formula}
\begin{split}
    \var(Y_r(t)) =& \frac{1}{(r!)^2 (2 \pi \bi)^{2r}} \oint \cdots \oint \left(\prod_{1 \leq i,j \leq r} \frac{1 - z_{j,2}/z_{i,1}}{1-t^{-1}z_{j,2}/z_{i,1}} - 1 \right) \\
    &\cdot \prod_{\ell=1}^2\left( \prod_{1 \leq i < j \leq r} (z_{i,\ell}-z_{j,\ell})^2 \prod_{s=1}^r \frac{1+t^{-1}z_{s,\ell}^{-1}}{z_{s,\ell}^{r-1}}e^{\tau z_{s,\ell}} \frac{d z_{s,\ell}}{z_{s,\ell}}\right).
    \end{split}
\end{align}
The contours in \eqref{eq:var_formula} are compact and independent of $t$, and due to the 
\[
\prod_{1 \leq i,j \leq r} \frac{1 - z_{j,2}/z_{i,1}}{1-t^{-1}z_{j,2}/z_{i,1}} - 1
\]
term the integrand converges to $0$ as $t \to 1$, which yields \eqref{eq:var_conv}. It remains to show \eqref{eq:exp_conv}. Taking all contours in \eqref{eq:exp_first} to be the unit circle so $z_i^{-1} = \bar{z_i}$, by rewriting the Weyl denominator 
\[
\prod_{1 \leq i < j \leq r} z_i - z_j = (-1)^{\binom{r}{2}} \left(\prod_{1 \leq i < j \leq r} \bar{z_i} - \bar{z_j}\right) \prod_{s=1}^r z_s^{r-1}
\]
we have
\begin{equation}\label{eq:exp_unitary_product}
   \E[Y_r(t)] = \frac{1}{r! (2 \pi \bi)^r} \oint \cdots \oint \prod_{1 \leq i < j \leq r} \abs{z_i - z_j}^2 \prod_{s=1}^r (1+t^{-1}\bar{z_s})e^{\tau z_s} \frac{dz_s}{z_s}.
\end{equation}
It is a classical fact, which follows from the Weyl character formula, Weyl integration formula and character orthogonality (or from the generalization to Macdonald polynomials in \cite[Chapter VI.9]{mac}), that the Schur polynomials $s_\l(z_1,\ldots,z_r) = P_\l(z_1,\ldots,z_r;t=0)$ are orthonormal with respect to the inner product
\begin{equation*}
    \langle f, g \rangle = \frac{1}{r! (2 \pi \bi)^r} \oint \cdots \oint \prod_{1 \leq i < j \leq r} \abs{z_i - z_j}^2 f(z_1,\ldots,z_r) \overline{g(z_1,\ldots,z_r)} \prod_{s=1}^r \frac{dz_s}{z_s}
\end{equation*}
where the integrals are over the unit circle in $\C$. Hence to compute \eqref{eq:exp_unitary_product} it suffices to expand $e^{\tau \sum_{s=1}^r z_s}$ and $\prod_{s=1}^r (1+t^{-1}\overline{z_s}) $ in terms of the Schur polynomials.
We have 
\begin{equation}\label{eq:bar_int_part}
    \prod_{s=1}^r (1+t^{-1}\bar{z_s}) = \sum_{k=0}^r t^{-k}\overline{e_k(z_1,\ldots,z_r)}
\end{equation}
where $e_k(z_1,\ldots,z_r) = s_{(1[k])}(z_1,\ldots,z_r)$ is the elementary symmetric polynomial, and 
\begin{equation}\label{eq:nobar_int_part}
    e^{\tau \sum_{s=1}^r z_s} = \sum_{j \geq 0} \frac{\tau^j}{j!} e_1(z_1,\ldots,z_r)^j.
\end{equation}
It follows from the classical Pieri rule for Schur functions, see for example \cite{mac}, that for $1 \leq j \leq r$
\begin{equation}\label{eq:schur_expansion}
    e_1(z_1,\ldots,z_r)^j = s_{(1[j])}(z_1,\ldots,z_r) + \ldots
\end{equation}
when expanded in the basis of Schur functions, where the other terms on the RHS of \eqref{eq:schur_expansion} are Schur functions $s_\l$ where $|\l|=j$ and $\l \neq (1[j])$. We therefore have 
\begin{equation}\label{eq:e_product}
   \langle e_k, e_1^j \rangle = \delta_{j,k}. 
\end{equation}
Combining \eqref{eq:bar_int_part}, \eqref{eq:nobar_int_part} and \eqref{eq:e_product} yields that 
\[
\text{RHS{\eqref{eq:exp_unitary_product}}} = \sum_{j=0}^r \frac{\tau^j}{j!}t^{-j}.
\]
Sending $t \to 1$ and tracing back the chain of equalities, this shows \eqref{eq:exp_conv} and hence completes the proof.
\end{proof}

\section{Gaussian fluctuations}\label{sec:fluctuations}

In this section we move on from the law of large numbers to study the fluctuations of the particle positions $x_k(\tau)$. \Cref{thm:gaussianity} uses general machinery of \cite{borodin2015general} to show Gaussian fluctuations for the particle positions $x_k\left(\tfrac{\tau}{1-t}\right)$ and gives a formula for their covariance, but the number of contour integrals in the formula grows with the particle index, making it intractable asymptotically. Taking a further $\tau \to \infty$ limit, this covariance converges (without rescaling) to an expression which can be simplified to a double contour integral with the aid of orthogonal polynomial techniques similar to those used in \cite[\S 5.1]{borodin2018anisotropic}. 

\begin{definition}\label{def:prelimit_fluc}
Letting $x_i$ be the position of the $i\tth$ particle of slowed $t$-TASEP with parameter $t = e^{-\eps}$, we define $X^{(i,\eps)}_\tau$ by 
\begin{equation*}
x_i\pfrac{\tau}{1-t} = \E\left[x_i\pfrac{\tau}{1-t}\right] + \eps^{-1/2}X^{(i,\eps)}_\tau.    
\end{equation*}
\end{definition}

\begin{proposition}\label{thm:gaussianity}
For any $n \in \Z_{\geq 1}$, the random vector $(X_\tau^{(1,\eps)},\ldots,X_\tau^{(n,\eps)})$ converges in distribution as $\eps \to 0^+$ to a mean $0$ Gaussian random vector $(X^{(1)}_\tau,\ldots,X^{(n)}_\tau)$. The covariances of these Gaussian random vectors are determined by the formula
\begin{multline}\label{eq:cov_with_vandermonde}
    \Cov(X^{(1)}_\tau+\ldots+X^{(r)}_\tau; X^{(1)}_\tau+\ldots + X^{(s)}_\tau) \\
    = \frac{\displaystyle \oint dz_{2,1} \cdots \oint dz_{2,s} \oint dz_{1,1} \cdots \oint dz_{1,r} \sum_{\substack{1 \leq i \leq r \\ 1 \leq j \leq s}}\frac{z_{2,j}}{z_{1,i}-z_{2,j}} F_\tau(z_{2,1},\ldots,z_{2,s})F_\tau(z_{1,1},\ldots,z_{1,r})}{\left(\displaystyle \oint dz_{2,1} \cdots \oint dz_{2,s} F_\tau(z_{2,1},\ldots,z_{2,s}) \right)\left(\displaystyle \oint dz_{1,1} \cdots \oint dz_{1,r} F_\tau(z_{1,1},\ldots,z_{1,r})\right)},
\end{multline}
for all $r \geq s \geq 1$, where 
\begin{equation}
    F_\tau(z_1,\ldots,z_k) = \Delta(z_1,\ldots,z_k)^2 \prod_{i=1}^k \frac{e^{\tau z_i}(1+z_i)}{z_i^{k+1}}
\end{equation}
and the contours are all positively oriented, encircle $0$ and satisfy $|z_{2,j}| < |z_{1,i}|$ for all $1 \leq i \leq r, 1 \leq j \leq s$.
\end{proposition}

\begin{remark}\label{rmk:formula_not_symmetric}
We note that the integrand in the formula for covariances \eqref{eq:cov_with_vandermonde} is not symmetric in $r,s$, and in fact the formula is \emph{not} valid if $r<s$. The same is true of the simplified formula \eqref{eq:final_prelimit_cov} which will be derived from it below in \Cref{thm:aux_limit_and_nicer_covariances}.
\end{remark}


\begin{proof}[Proof of {\Cref{thm:gaussianity}}]
Since $x_i\left(\tfrac{\tau}{1-t}\right) = \lambda_i'(\tau) - i$, 
\begin{equation*}
    X_\tau^{(i,\eps)} = \eps^{1/2}\left(\lambda_i'(\tau) - \E\left[\lambda_i'(\tau)\right]\right)
\end{equation*} 
Clearly it suffices to show that the family of random variables $(X_\tau^{(1,\eps)} + \ldots + X_\tau^{(r,\eps)})_{r \geq 1}$ converge jointly to the Gaussian family $(X_\tau^{(1)} + \ldots + X_\tau^{(r)})_{r \geq 1}$. We will first show that another family of random variables 
\[
V_r(\eps,\tau) := t^{-(\l_1'(\tau)+\ldots+\l_r'(\tau))}, r \geq 1
\]
converges jointly to the Gaussian family $(X_\tau^{(1)} + \ldots + X_\tau^{(r)})_{r \geq 1}$ after appropriate scaling, and then argue this suffices.

\Cref{thm:observable_formula} gives us contour integral formulas for all joint moments of these random variables, so it is a matter of analyzing these integral formulas. We will use the general Gaussianity lemma given as Lemma $4.2$ of \cite{borodin2015general}, which has a self-contained presentation in Section $4.3$ of the same paper.

Let
\begin{align}
    Cr_\eps(r_\alpha,\bm{z}_\alpha; r_\beta, \bm{z}_\beta) &= \prod_{\substack{ 1 \leq i \leq r_\alpha \\ 1 \leq j \leq r_\beta}}\frac{1-z_{j,\beta}/z_{i,\alpha}}{1-t^{-1}z_{j,\beta}/z_{i,\alpha}} \\
    \mf F_\eps(r_m,\bm{z}_m) &= \frac{(-1)^{\binom{r_m}{2}}}{r_m!(2 \pi i)^{r_m}}\Delta(z_{1,m},\ldots,z_{r_m,m})^2 \prod_{s=1}^{r_m} 
    \frac{1+t^{-1}z_{s,m}^{-1}}{z_{s,m}^{r_m}}e^{\tau z_{s,m}}
\end{align}
where $\bm{z}_\alpha$ is shorthand for the tuple of variables $z_{1,\alpha},\ldots,z_{r_\alpha,\alpha}$, so that \Cref{thm:observable_formula} reads
\begin{equation}
    \E[V_{r_1}(\eps,\tau) \cdots V_{r_m}(\eps,\tau)] = \oint \cdots \oint \prod_{1 \leq \alpha < \beta \leq k} Cr_\eps(r_\alpha,\bm{z}_\alpha; r_\beta, \bm{z}_\beta ) \prod_{s=1}^k {\mf F}_\eps(r_m,\bm{z}_m) d\bm{z}_m.
\end{equation}
We have
\begin{equation}
    Cr_\eps(r_\alpha,\bm{z}_\alpha; r_\beta, \bm{z}_\beta) = 1 + \eps \mf{Cr}_\eps(r_\alpha,\bm{z}_\alpha; r_\beta, \bm{z}_\beta),
\end{equation}
and uniform convergence $\mf{Cr}_\eps \to \mf{Cr}$ and $\mf{F_\eps} \to \mf{F}$ along the contours of interest, where
\begin{align}
    \mf{Cr}(r_\alpha,\bm{z}_\alpha; r_\beta, \bm{z}_\beta) &= \sum_{\substack{ 1 \leq i \leq r_\alpha \\ 1 \leq j \leq r_\beta}} \frac{z_{j,\beta}}{z_{i,\alpha}-z_{j,\beta}} \\
    \mf F(r_\alpha,\bm{z}_\alpha) &= \frac{(-1)^{\binom{r_m}{2}}}{r_m!(2 \pi i)^{r_m}}\Delta(z_{1,\alpha},\ldots,z_{r_\alpha,\alpha})^2 \prod_{s=1}^{r_\alpha} \frac{e^{\tau z_{s,\alpha}}(1+z_{s,\alpha})}{z_{s,\alpha}^{r_\alpha+1}}.
\end{align}
By \cite[Lemma 4.2]{borodin2015general}\footnote{In the notation of \cite{borodin2015general} one should take $\eps = L^{-1}$ and $\gamma = 1$.}, this implies that the random variables 
\begin{equation}
    \label{eq:def_Y}
    Y_r(\eps,\tau) := \frac{V_r(\eps,\tau) - \E[V_r(\eps,\tau)]}{\sqrt{\eps}}
\end{equation}
converge jointly to the mean $0$, jointly Gaussian family $(Y_r(\tau))_{r \geq 1}$ having covariance
\begin{align}\label{eq:postlimit_cov}
\begin{split}
    &\Cov(Y_r(\tau); Y_s(\tau)) \\
    &= \frac{\displaystyle \oint dz_{2,1} \cdots \oint dz_{2,s} \oint dz_{1,1} \cdots \oint dz_{1,r} \mf{Cr}(r,\bm{z}_1; s, \bm{z}_2) \mf F(r,\bm{z}_1) \mf F(s,\bm{z}_2)}{\left(\displaystyle \oint dz_{2,1} \cdots \oint dz_{2,s} \mf F(s,\bm{z}_2) \right)\left(\displaystyle \oint dz_{1,1} \cdots \oint dz_{1,r} \mf F(r,\bm{z}_1)\right)}.
\end{split}
\end{align}
After cancelling the $\frac{(-1)^{\binom{r_m}{2}}}{r_m!(2 \pi i)^{r_m}}$ terms in the numerator and denominator this is exactly the RHS of \eqref{eq:cov_with_vandermonde}, hence we indeed have that $(Y_r(\eps,\tau))_{r \geq 1}$ converges jointly to $(X_\tau^{(1)} + \ldots + X_\tau^{(r)})_{r \geq 1}$. 

It remains to show that $(X^{(1,\eps)}_\tau+\ldots+X^{(r,\eps)}_\tau)_{r \geq 1}$ also converges jointly to $(X_\tau^{(1)} + \ldots + X_\tau^{(r)})_{r \geq 1}$. At a heuristic level this makes perfect sense by Taylor expanding the exponential in 
\[
V_r(\eps,\tau) = e^{\eps \left(\E[\l_1'(\tau)+\ldots+\l_r'(\tau)] + \eps^{1/2}\left(X^{(1,\eps)}_\tau+\ldots+X^{(r,\eps)}_\tau\right)\right)}
\]
in \eqref{eq:def_Y}, as the leading-order nonconstant term is $\text{const}\cdot (X^{(1,\eps)}_\tau+\ldots+X^{(r,\eps)}_\tau)$ and the others are small in $\eps$. To make this rigorous one uses (joint) tightness in $\eps$ of the random variables $Y_r(\eps,\tau)$ to show joint tightness of the random variables $X^{(1,\eps)}_\tau+\ldots+X^{(r,\eps)}_\tau$, which are related by a simple transformation, and then argues using Prokhorov's theorem, the convergence of $Y_r(\eps,\tau)$ and the previous Taylor expansion. The details are given in the proof of Proposition $4.1$ in \cite{borodin2018anisotropic}, where our $Y_r(\eps,\tau)$ corresponds to their $Y_r^\eps$, the analogue of the Gaussian convergence for $Y_r(\eps,\tau)$ is Lemma $4.4$ of \cite{borodin2018anisotropic}, and with these two substitutions the proof carries over mutatis mutandis in our setting.
\end{proof}

We note that for a family of random variables with an only slightly different integral formula for covariances, the analogue of the above Gaussian convergence argument is written in a self-contained manner in the proof of \cite[Proposition 4.1]{borodin2018anisotropic}. For a reader wishing to understand all the details of the proof, this might be easier to read than the proof in \cite{borodin2015general} of the general Gaussianity lemma used in our condensed version. 


Returning to our setting, the formula in \Cref{thm:gaussianity} simplifies greatly upon taking another limit $\tau \to \infty$. As was argued in the introduction, and will be fleshed out in the next section, this limit reflects convergence to stationarity of the original particle system (with an additional time change).

\begin{proposition}\label{thm:aux_limit_and_nicer_covariances}
As $\tau \to  \infty$, the random variables $X^{(i)}_\tau$ converge in distribution to a Gaussian random vector $(\zeta_i)_{i \geq 1}$, with covariances given by 
\begin{equation}\label{eq:final_prelimit_cov}
    \Cov(\zeta_r;\zeta_s) = \frac{1}{4 \pi^2} \displaystyle \oint_{\Gamma_0} \oint_{\Gamma_{0,w}} \frac{w}{z-w} \frac{r! s!}{z^{r}w^{s}}e^{z+w}(1-z/r)(1-w/s)\frac{dz}{z} \frac{dw}{w}
\end{equation}
for each $r \geq s \geq 1$.
\end{proposition}

\begin{proof}
First note that by symmetry of $F_\tau$, we may replace 
\begin{equation*}
    \sum_{\substack{1 \leq i \leq r \\ 1 \leq j \leq s}}\frac{z_{j,2}}{z_{i,1}-z_{j,2}}
\end{equation*}
by 
\begin{equation*}
    r s \frac{z_{1,2}}{z_{1,1}-z_{1,2}}
\end{equation*}
in \eqref{eq:cov_with_vandermonde}. Now, changing variables to $z_i = \tau z_{i,1}, w_j = \tau z_{j,2}$ and cancelling the factors of $\tau$ that appear, \eqref{eq:cov_with_vandermonde} becomes
\begin{multline}\label{eq:cov_with_vandermonde2}
    \Cov(X^{(1)}_\tau+\ldots+X^{(r)}_\tau; X^{(1)}_\tau+\ldots + X^{(s)}_\tau) \\ = r s \frac{\displaystyle \oint  \cdots \oint  \frac{w_1}{z_1-w_1} \Delta(\bm{z})^2 \prod_{i=1}^r \frac{e^{z_i}(1+z_i/\tau)}{z_i^{r+1}} \Delta(\bm{w})^2 \prod_{j=1}^s \frac{e^{w_j}(1+w_j/\tau)}{w_j^{s+1}} d\bm{z}d\bm{w}}{\left(\displaystyle \oint \cdots \oint \Delta(\bm{z})^2 \prod_{i=1}^r \frac{e^{z_i}(1+z_i/\tau)}{z_i^{r+1}} d\bm{z} \right)\left(\displaystyle \oint \cdots \oint\Delta(\bm{w})^2 \prod_{j=1}^s \frac{e^{w_j}(1+w_j/\tau)}{w_j^{s+1}} d\bm{w}\right)}.
\end{multline}
Since 
\begin{equation*}
    \Delta(z_1,\ldots,z_r)^2 \prod_{i=1}^r \frac{e^{z_i}(1+z_i/\tau)}{z_i^{r+1}} \to \Delta(z_1,\ldots,z_r)^2 \prod_{i=1}^r \frac{e^{z_i}}{z_i^{r+1}} 
\end{equation*}
uniformly on the contours of integration, the RHS of \eqref{eq:cov_with_vandermonde2} converges as $\tau \to  \infty$ to the same expression with the $(1+z_i/\tau)$ and $(1+w_j/\tau)$ factors removed. In particular, because the $X^{(i)}_\tau$ are Gaussian, this implies convergence in joint distributions $X^{(i)}_\tau \to \zeta_i$, where the $\zeta_i$ form Gaussian random vectors with covariances given by
\begin{multline}
    \Cov(\zeta_1+\ldots+\zeta_r; \zeta_1+\ldots+\zeta_s) \\
    = r s \frac{\displaystyle \oint  \cdots \oint  \frac{w_1}{z_1-w_1} \Delta(z_1,\ldots,z_r)^2 \prod_{i=1}^r \frac{e^{z_i}}{z_i^{r+1}} \Delta(w_1,\ldots,w_s)^2 \prod_{j=1}^s \frac{e^{w_j}}{w_j^{s+1}} d\bm{z}d\bm{w}}{\left(\displaystyle \oint \cdots \oint \Delta(z_1,\ldots,z_r)^2 \prod_{i=1}^r \frac{e^{z_i}}{z_i^{r+1}} d\bm{z} \right)\left(\displaystyle \oint \cdots \oint\Delta(w_1,\ldots,w_s)^2 \prod_{j=1}^s \frac{e^{w_j}}{w_j^{s+1}} d\bm{w}\right)}.
\end{multline}
Rewriting the above as 
\begin{equation}\label{eq:cov_zeta_sum}
    \Cov(\zeta_1+\ldots+\zeta_r; \zeta_1+\ldots+\zeta_s) = \frac{1}{(2 \pi i)^2} \displaystyle \oint \oint \frac{w_1}{z_1-w_1} \rho_r(z_1) \rho_s(w_1) dz_1 dw_1
\end{equation}
where
\begin{equation}
    \rho_r(z_1) = \frac{r \frac{1}{(2 \pi i)^{r-1}}\displaystyle \oint \cdots \oint  \Delta(z_1,\ldots,z_r)^2 \prod_{i=1}^r \frac{e^{z_i}}{z_i^{r+1}}   dz_2 \cdots dz_r}{\frac{1}{(2 \pi i)^r} \displaystyle \oint \cdots \oint  \Delta(z_1,\ldots,z_r)^2 \prod_{i=1}^r \frac{e^{z_i}}{z_i^{r+1}}  dz_1 dz_2 \cdots dz_r},
\end{equation}
we recognize $\rho_r(z)$ as the $1$-point correlation function of the orthogonal polynomial ensemble on the contour $\Gamma_0$ with weight $\Delta(z_1,\ldots,z_r)^2 \prod_{i=1}^r \frac{e^{z_i}}{z_i^{r+1}}$. 

Let $p_k^{n}$ be the (monic) orthogonal polynomial of degree $k$ with respect to the inner product
\begin{equation}
    \la f, g \ra_{n} = \frac{1}{2 \pi i} \displaystyle \oint f(z)g(z) \frac{e^z}{z^{n}}dz.
\end{equation}
Then by the classical theory of orthogonal polynomials, see e.g. \cite{deift1999orthogonal}, one has
\begin{equation}\label{eq:rho_sum}
    \rho_r(z) = \frac{e^z}{z^{r+1}}\sum_{k=0}^{r-1} \frac{p_k^{r+1}(z)^2}{\la p_k^{r+1},p_k^{r+1} \ra_{r+1}}.
\end{equation}
This reduces the computation of \eqref{eq:cov_zeta_sum} to understanding the orthogonal polynomials $p_k^{r+1}$. For the observation above that $\rho_r(z)$ is a $1$-point correlation function we followed a similar argument in \cite{borodin2018anisotropic}, and in fact our orthogonal polynomial ensemble is a special case of the one in that paper. They prove\footnote{To be specific, one must specialize $T=1$ in the notation of \cite[Lemma 5.3]{borodin2018anisotropic} to arrive at Lemma \ref{thm:formulas_from_bf}.} the following explicit formulas by relating the $p_k^n$ to the classical Laguerre polynomials, for which similar explicit formulas are classically known.

\begin{lemma}[{\cite[Lemma 5.3]{borodin2018anisotropic}}]\label{thm:formulas_from_bf}
Let $p_k^n$ be as above. Then 
\begin{equation}\label{eq:pk_formulas}
    p_k^n(z) = \frac{k!}{(n-k-1)!} \sum_{\ell=0}^k \frac{(n-1-\ell)!}{(k-\ell)!\ell!}(-z)^\ell = \frac{1}{(n-1-k)!} \int_0^\infty (y-z)^k y^{n-1-k} e^{-y}dy.
\end{equation}
Furthermore
\begin{equation}
    \la p_k^{n},p_k^{n} \ra_{n} = (-1)^k \frac{k!}{(n-k-1)!}.
\end{equation}
\end{lemma}

We rewrite \eqref{eq:rho_sum} as 
\begin{equation}\label{eq:sum_and_pair}
    \rho_r(z) = \left(\frac{e^z}{z^{r+1}}\sum_{k=0}^{r} \frac{p_k^{r+1}(z)^2}{\la p_k^{r+1},p_k^{r+1} \ra_{r+1}}\right) - \frac{e^z}{z^{r+1}}\frac{p_r^{r+1}(z)^2}{\la p_r^{r+1},p_r^{r+1} \ra_{r+1}}
\end{equation}
and treat the two terms on the RHS separately. We first treat the sum on the RHS of \eqref{eq:sum_and_pair}, which will end up not contributing at all. By Lemma \ref{thm:formulas_from_bf},
\begin{align*}
    \sum_{k=0}^{r} \frac{p_k^{r+1}(z)^2}{\la p_k^{r+1},p_k^{r+1} \ra_{r+1}} &= \sum_{k=0}^{r} \frac{(-1)^k}{(r-k)!k!} \int_0^\infty (y-z)^k y^{r-k} e^{-y}dy \int_0^\infty (x-z)^k x^{r-k} e^{-y}dx \\
    &= \frac{1}{r!} \int_0^\infty \int_0^\infty \sum_{k=0}^r \binom{r}{k}(-(y-z)(x-z))^k (xy)^{r-k} e^{-(x+y)}dx dy \\
    &= \frac{1}{r!} \int_0^\infty \int_0^\infty \left(z(x+y)-z^2\right)^r e^{-(x+y)}dx dy \\
    &= \frac{z^r}{r!} \int_0^\infty \int_0^u (u-z)^r e^{-u} du dy \\
    &= \frac{z^r}{r!} \cdot \left(e^{-z} ((r+1)!+r! z) + O(z^{r+1})\right).
\end{align*}
Hence
\begin{equation}\label{eq:sum=simple}
   \frac{e^z}{z^{r+1}}\sum_{k=0}^{r} \frac{p_k^{r+1}(z)^2}{\la p_k^{r+1},p_k^{r+1} \ra_{r+1}} = \frac{r+1}{z} + 1 + O(z^r).
\end{equation}
By \eqref{eq:pk_formulas}, 
\begin{equation}
    p_r^{r+1}(z) = r! \sum_{\ell=0}^r \frac{(-z)^\ell}{\ell!} = r!\left(e^{-z} - f_{r+1}(z)\right),
\end{equation}
where $f_{r+1}(z) := \sum_{\ell=r+1}^\infty \frac{(-z)^\ell}{\ell!}$ is just the sum of terms of degree $\geq r+1$ in the Taylor series for $e^{-z}$. Hence 
\begin{multline}
    \left(e^{-z} - f_{r+1}(z)\right)^2 = e^{-2z} - 2 e^{-z}f_{r+1}(z) + O(z^{2r+2}) \\
    = e^{-z}\left(\sum_{\ell=0}^\infty (-1)^{\bbone(\ell \geq r+1)} \frac{(-z)^\ell}{\ell!} + O(z^{2r+2})\right),
\end{multline}
so 
\begin{equation}\label{eq:p_product_calc}
  \frac{e^z}{z^{r+1}}  \frac{p_r^{r+1}(z)^2}{\la p_r^{r+1},p_r^{r+1} \ra_{r+1}} =  \frac{r!}{(-z)^{r+1}}\sum_{\ell=0}^\infty (-1)^{\bbone(\ell \geq r+1)} \frac{(-z)^{\ell}}{\ell!} + O(z^{r+1}).
\end{equation}
Substituting \eqref{eq:sum=simple} and \eqref{eq:p_product_calc} into \eqref{eq:sum_and_pair} yields
\begin{equation}\label{eq:rho_nice_form}
    \rho_r(z) = \frac{r+1}{z} + 1  - \frac{r!}{(-z)^{r+1}}\sum_{\ell=0}^\infty (-1)^{\bbone(\ell \geq r+1)} \frac{(-z)^{\ell}}{\ell!} + O(z^{r}).
\end{equation}
Recall that 
\begin{equation}\label{eq:cov_zeta_sum2}
    \Cov(\zeta_1+\ldots+\zeta_r; \zeta_1+\ldots+\zeta_s) = \frac{1}{(2 \pi i)^2} \displaystyle \oint_{\Gamma_0} \oint_{\Gamma_{0,w}} \frac{w}{z-w} \rho_r(z) \rho_s(w) dz dw.
\end{equation}
Since $|w| < |z|$ in the region of integration we may expand $\frac{w}{z-w} = \sum_{n=1}^\infty \pfrac{w}{z}^n$ in the integrand, and then interpret the integral as the $\frac{1}{zw}$ term of the resulting Laurent series expansion for the integrand (this may be justified by applying the residue theorem first to the $z$ integral, then the $w$ integral). Since $n$ is positive in the terms $\pfrac{w}{z}^n$, this yields that only the terms of $\rho_s(w)$ of degree $\leq -2$ in $w$ contribute, and only the terms of $\rho_r(z)$ of degree $\geq 0$ contribute. The terms of degree $\leq -2$ in $\rho_s(w)$ match those of $\frac{s!}{(-w)^{s+1}}e^{-w}$, so we may substitute this for $\rho_s(w)$ in \eqref{eq:cov_zeta_sum2} without changing the integral. Because all terms in the Laurent expansion for $\rho_s(w)$ have degree $\geq -(s+1)$, we additionally have that only the terms of $\rho_r(z)$ of degree $\leq s-1$ contribute. Because $r \geq s$, we may thus ignore the $O(z^r)$ terms in \eqref{eq:rho_nice_form}. The terms of degree $0 \leq d \leq s-1$ in the Laurent expansion for $\rho_r(z)$ in \eqref{eq:rho_nice_form} are the same as those in the Laurent series expansion of $-\frac{r!}{(-z)^{r+1}}e^{-z} + 1$. Therefore
\begin{equation}\label{eq:cov_zeta_sum3}
    \Cov(\zeta_1+\ldots+\zeta_r; \zeta_1+\ldots+\zeta_s) = \frac{1}{(2 \pi i)^2} \displaystyle \oint_{\Gamma_0} \oint_{\Gamma_{0,w}} \frac{w}{z-w} \frac{s!e^{-w}}{(-w)^{s+1}}\left(-\frac{r!}{(-z)^{r+1}}e^{-z} + 1\right)  dz dw.
\end{equation}

Denoting the RHS above by $C(r,s)$, we have $\Cov(\zeta_r;\zeta_s) = C(r,s) - C(r-1,s) - C(r,s-1) + C(r-1,s-1)$. Writing
\begin{multline}
   C(r,s) =  \frac{1}{(2 \pi i)^2} \displaystyle \oint_{\Gamma_0} \oint_{\Gamma_{0,w}} \frac{w}{z-w} \frac{s!}{(-w)^{s+1}}e^{-w}\left(-\frac{r!}{(-z)^{r+1}}e^{-z}\right)  dz dw \\
   + \frac{1}{(2 \pi i)^2} \displaystyle \oint_{\Gamma_0} \oint_{\Gamma_{0,w}} \frac{w}{z-w} \frac{s!}{(-w)^{s+1}}e^{-w}  dz dw,
\end{multline}
we see that the second integral on the RHS is independent of $r$, hence its contribution cancels in $C(r,s) - C(r-1,s) - C(r,s-1) + C(r-1,s-1)$. Thus
\begin{multline}
   C(r,s) - C(r-1,s) - C(r,s-1) + C(r-1,s-1) = -\frac{1}{(2 \pi i)^2} \displaystyle \oint_{\Gamma_0} \oint_{\Gamma_{0,w}} \frac{w}{z-w} e^{-z-w} \\
   \left(\frac{r!s!}{(-z)^{r+1}(-w)^{s+1}} - \frac{(r-1)!s!}{(-z)^{r}(-w)^{s+1}} - \frac{r!(s-1)!}{(-z)^{r+1}(-w)^{s}}+\frac{(r-1)!(s-1)!}{(-z)^{r}(-w)^{s}}\right)dzdw \\
   = \frac{1}{4 \pi^2} \displaystyle \oint_{\Gamma_0} \oint_{\Gamma_{0,w}} \frac{w}{z-w} \frac{r! s!}{(-z)^{r+1}(-w)^{s+1}}e^{-z-w}(1+z/r)(1+w/s)dz dw.
\end{multline}
Changing variables to $-z, -w$ yields \eqref{eq:final_prelimit_cov}, completing the proof.
\end{proof}

\section{Long-time SDEs and stationarity of fluctuations}\label{sec:SDEs}

In the limit $t=e^{-\eps}, \text{time}=\tau/(1-t)$, we previously derived Gaussian fluctuations $X^{(k)}_\tau$ for the particle positions, with explicit covariances which simplify in the large-time limit $\tau \to \infty$. In this section, we consider the probabilistic meaning of this limit. For particle systems such as ours, one may rigorously show convergence of the multi-time fluctuations to the solution of a system of SDEs as in \cite[Theorem 1]{borodin2017stochastic}, but we will instead give a (simpler and more intuitive) formal derivation of such a system of SDEs which closely follows that of \cite[Proposition 4.6]{borodin2018anisotropic}. After making a time change 
\[
Z_T^{(k)} := X_{e^T-1}^{(k)},
\]
this yields a system of SDEs for the $Z_T^{(k)}$ with time-dependent coefficients. As $T \to \infty$ these coefficients converge to nontrivial limits, yielding the system of SDEs
\begin{equation}\label{eq:final_SDEs}
    dZ^{(k)}_T = \left((k-1)Z^{(k-1)}_T - k Z^{(k)}_T\right)dT + dW^{(k)}_T \quad \quad \quad k=1,2,\ldots
\end{equation}
Though the derivation of the SDEs satisfied by the $X^{(k)}_\tau$ consisted of formal algebraic manipulations and is certainly not a rigorous analytic proof, we will check rigorously in \Cref{thm:stationarity} using contour integral formulas that the unique Gaussian stationary distribution of the system \eqref{eq:final_SDEs} is in fact the single-time limit of the fluctuations derived in the previous section.

We now proceed to the heuristic derivation of SDEs for $X^{(k)}_\tau, k=1,2,\ldots$. Because each particle jumps according to an independent Poisson clock with rate depending only on its position and that of the particle in front, the fluctuations should satisfy an SDE of the form 
\[
dX_\tau^{(k)} = f(\tau,X_\tau^{(k-1)},X_\tau^{(k)})d\tau + g(\tau,X_\tau^{(k-1)},X_\tau^{(k)})dB^{(k)}_\tau,
\]
and it remains to compute the drift and diffusion coefficients. To find the drift $f(\tau,X_\tau^{(k-1)},X_\tau^{(k)})$, we take expectations of both sides to eliminate the diffusion part. Hence we must compute the $\eps \to 0$ limit of the $O(d\tau)$ term in 
\begin{equation}\label{eq:drift_term}
\E\left[X^{(k,\eps)}_{\tau+d\tau}-X^{(k,\eps)}_\tau\right] = -\eps^{-1/2}(c_k(\tau+d\tau)-c_k(\tau)) + \eps^{1/2}\E\left[ \lambda_k'(\tau+d\tau) - \lambda_k'(\tau)\right],
\end{equation}
where $c_k(\tau)$ is the limit of $\eps x_k(\tau/\eps)$ given explicitly in \eqref{eq:def_c}. The jump rate of $\lambda_k'(\tau)$ is approximately constant on the interval $d\tau$, and equal to 
\begin{equation}\label{eq:jump_rate}
 \frac{1}{1-t}\left(t^{\l_k'(\tau)} - t^{\l_{k-1}'(\tau)}\right) \sim \eps^{-1} \left(e^{-(c_k(\tau) + \eps^{1/2}X^{(k,\eps)}_\tau)} - e^{-(c_{k-1}(\tau) + \eps^{1/2}X^{(k-1,\eps)}_\tau)}\right) \quad \quad \text{ as $\eps \to 0$}   
\end{equation}
where to obtain the RHS we use that $1-t \approx \eps$. Therefore 
\begin{align}\label{eq:rate_2nd_term}
\begin{split}
    &\eps^{1/2}\E\left[ \lambda_k'(\tau+d\tau) - \lambda_k'(\tau)\right] \sim \eps^{-1/2}d\tau\left(e^{-(c_k(\tau) + \eps^{1/2}X^{(k,\eps)}_\tau)} - e^{-(c_{k-1}(\tau) + \eps^{1/2}X^{(k-1,\eps)}_\tau)} \right) \\
    & \sim \eps^{-1/2}d\tau \left(e^{-c_k(\tau)} - e^{-c_{k-1}(\tau)} \right) - \left(e^{-c_k(\tau)}X^{(k,\eps)}_\tau - e^{-c_{k-1}(\tau)}X^{(k-1,\eps)}_\tau \right) 
    \end{split}
\end{align}
as $\eps \to 0$. The other term on the RHS of \eqref{eq:drift_term} is 
\begin{align}\label{eq:cprime}
\begin{split}
    -\eps^{-1/2}(c_k(\tau+d\tau)-c_k(\tau)) &= -\eps^{-1/2}c_k'(\tau)d\tau + O(d\tau^2) \\
    &= -\eps^{-1/2}\left(e^{-c_k(\tau)} - e^{-c_{k-1}(\tau)} \right)d\tau + O(d\tau^2)
\end{split}
\end{align}
by the differential equation \eqref{eq:c_ODE}. Combining \eqref{eq:rate_2nd_term} and \eqref{eq:cprime} yields a term which converges as $\eps \to 0$, hence the drift coefficient is
\begin{equation}\label{eq:final_drift_coef}
    f(\tau,X_\tau^{(k-1)},X_\tau^{(k)}) = \lim_{\eps \to 0} \text{RHS{\eqref{eq:drift_term}}} = - \left(e^{-c_k(\tau)}X^{(k)}_\tau - e^{-c_{k-1}(\tau)}X^{(k-1)}_\tau \right).
\end{equation}
We now compute the diffusion coefficient, which is the $O(d\tau)$ term in
\begin{equation}
    \label{eq:var_X_lam}
    \operatorname{Var}(X^{(k,\eps)}_{\tau+d\tau}-X^{(k,\eps)}_\tau) = \eps \operatorname{Var}(\l'_k(\tau+d\tau) - \l'_k(\tau)).
\end{equation}
We approximate the jump rate to be constant as before, so that 
\[
\l'_k(\tau+d\tau) - \l'_k(\tau)
\]
is a Poisson random variable with parameter equal to the time step $d\tau$ times its jump rate approximated earlier in \eqref{eq:jump_rate}. Since variance of $\text{Pois}(r)$ is $r$, 
\begin{equation}
    \text{RHS}\eqref{eq:var_X_lam} = \left(e^{-c_k(\tau)} - e^{-c_{k-1}(\tau)}\right)d\tau + o(1).
\end{equation}
Hence 
\begin{equation}\label{eq:diffusion_final_formula}
    g(\tau,X_\tau^{(k-1)},X_\tau^{(k)}) = \sqrt{e^{-c_k(\tau)} - e^{-c_{k-1}(\tau)}}.
\end{equation}
Combining \eqref{eq:final_drift_coef} with \eqref{eq:diffusion_final_formula}, we have derived (again, at a heuristic level) that the $\eps \to 0$ limits $X^{(k)}_\tau$ satisfy the system 
\begin{equation}\label{eq:X_SDE}
    dX^{(k)}_\tau = - \left(e^{-c_k(\tau)}X^{(k)}_\tau - e^{-c_{k-1}(\tau)}X^{(k-1)}_\tau \right)d\tau + \sqrt{e^{-c_k(\tau)} - e^{-c_{k-1}(\tau)}} dB^{(k)}_\tau \quad \quad k=1,2,\ldots 
\end{equation}
where as before we take $c_0(\tau) \equiv \infty$ identically in the case $k=1$. 

Exponentiating the explicit formula \eqref{eq:def_c} for $c_k(\tau)$ yields
\begin{equation}\label{eq:recall_ck}
    e^{-c_k(\tau)} = \frac{1+\ldots+\frac{\tau^{k-1}}{(k-1)!}}{1+\ldots+\frac{\tau^k}{k!}}.
\end{equation}

Naively taking the $\tau \to \infty$ limit of the diffusion coefficient in \eqref{eq:X_SDE} yields $0$, which reflects the fact that particles' jump rates go to $0$ as their positions go to $\infty$ due to the position-dependent slowing. However, the prelimit system also suggests a natural time change to obtain time-independent diffusion rates. The jump rate of $\l_1'$ is $t^{\l_1'}$, so to make this jump rate independent of time one must speed up time by a factor of $t^{-\l_1'}$---which, note, depends on the random position of $\l_1'$. More precisely, if $s$ is the time variable in the original particle system, then letting $h(s)$ be the piecewise-linear random function with $h'(s) = t^{-\l_1'(h(s))}$, one has that $\l_1'(h(s))$ jumps according to a rate $1$ Poisson process. Hence its position is a Poisson random variable with mean $s$. Since it concentrates around its mean at large $s$, we have $h'(s) \approx t^{-s}$ and hence
\[
h(s) \approx \frac{t^{-s}}{-\log t}
\]
for large $s$. This suggests that the random time change by $t^{-\l_1'}$ can be approximated at large times by a deterministic exponential time change, so we make an exponential time change $\tau = e^T$ in the limit SDEs \eqref{eq:X_SDE}. For notational convenience let us instead shift slightly and take $\tau = e^T-1$ so that $T$ begins at $0$. Setting $Z_T^{(k)} := X^{(k)}_{e^T-1}$ in \eqref{eq:X_SDE}, one has $d\tau = e^T dT$ and $dB^{(k)}_T = \sqrt{e^T} dW^{(k)}_T$ for $W^{(k)}_T$ independent standard Brownian motions, yielding
\begin{equation}\label{eq:Z_SDE}
    dZ^{(k)}_T = - \left(e^{-c_k(e^T-1)}Z^{(k)}_T - e^{-c_{k-1}(e^T-1)}Z^{(k-1)}_T\right)e^T dT + \sqrt{e^T\left(e^{-c_k(e^T-1)} - e^{-c_{k-1}(e^T-1)}\right)} dW^{(k)}_T
\end{equation}
Plugging in \eqref{eq:recall_ck} we obtain 
\[
dZ^{(k)}_T = \left( -kZ^{(k)}_T + (k-1)Z^{(k-1)}_T + o(1)\right) dT + (1+o(1)) dW^{(k)}_T \quad \quad k=1,2,\ldots 
\]
which converges to \eqref{eq:final_SDEs}. This mirrors the convergence of the covariances of particle fluctuations without rescaling as $\tau \to \infty$, shown in \Cref{thm:aux_limit_and_nicer_covariances}, and the main result of this section is that the SDEs \eqref{eq:final_SDEs} indeed have a stationary solution with the exact covariances of \Cref{thm:aux_limit_and_nicer_covariances}. 

We note also for concreteness that the SDE for $Z^{(1)}_T$ is exactly that of an Ornstein-Uhlenbeck process, and the mean-reversion reflects the fact that the jump rate of $\l_1'$ is smaller when it is further ahead and larger when it is further behind. The dependence of the drift term on $Z^{(k)}_T, Z^{(k-1)}_T$ likewise reflects the prelimit dependence of a particle's jump rate on its own position and that of the particle in front.  

Let us now proceed rigorously. We first check that it makes sense to speak of \emph{the} solution to \eqref{eq:final_SDEs}.

\begin{lemma}
    \label{thm:SDE_uniqueness}
Strong existence and uniqueness hold for the system of SDEs
\[
dZ^{(k)}_T = \left((k-1)Z^{(k-1)}_T - k Z^{(k)}_T\right)dT + dW^{(k)}_T \quad \quad \quad T \geq 0 , k=1,2,\ldots
\]
stated earlier as \eqref{eq:final_SDEs}. 
\end{lemma}
\begin{proof}
    Note that for each $n \geq 1$, the coefficients in the SDEs \eqref{eq:final_SDEs} for $(Z^{(k)}_T)_{1 \leq k \leq n}$ depend only on $(Z^{(k)}_T)_{1 \leq k \leq n}$, i.e. $(Z^{(1)}_T,\ldots,Z^{(n)}_T)$ satisfies an SDE
    \begin{equation}
        \label{eq:finite_sde}
        dZ^{(k)}_T = \left((k-1)Z^{(k-1)}_T - k Z^{(k)}_T\right)dT + dW^{(k)}_T \quad \quad \quad k=1,\ldots,n.
    \end{equation}
    in $\R^n$ driven by noise $(W^{(1)}_T,\ldots,W^{(n)}_T)$. We claim it suffices to prove strong existence and uniqueness of \eqref{eq:finite_sde} for each $n$, which we recall means that given $(Z^{(k)}_0)_{1 \leq k \leq n}$ and a fixed Brownian motion $(W^{(k)}_T)_{1 \leq k \leq n}$, there is a process solving \eqref{eq:final_SDEs} which is unique up to almost-everywhere equivalence. The claim holds because the resulting $n$-indexed family of solutions is clearly consistent under forgetting the last coordinate $Z^{(n)}_T$, hence the consistent $n$-indexed family defines a solution $(Z^{(k)}_T)_{k \geq 1}$ to the infinite system \eqref{eq:final_SDEs}.
    
    We now argue for fixed $n$ by applying off-the-shelf existence and uniqueness theorems. To aid in matching notation, let
    \begin{align*}
        b_k(T,\bm{x}) &= (k-1) x_{k-1} - k x_k \\
        \bm{b}(T,\bm{x}) &= (b_1(T,\bm{x}),\ldots,b_n(T,\bm{x})) \\
        \sigma_{ij}(T,\bm{x}) &= \bbone(i=j)
    \end{align*}
    for $T \geq 0, \bm{x} \in \R^n$, so that \eqref{eq:final_SDEs} takes the form
    \[
dZ^{(k)}_T = b_k(T,(Z^{(1)}_T,\ldots,Z^{(n)}_T))dT + \sum_{\ell=1}^n \sigma_{k \ell}(T,(Z^{(1)}_T,\ldots,Z^{(n)}_T))dW^{(\ell)}_T.
    \]
    For strong uniqueness, by \cite[Chapter 5.2, Theorem 2.5]{karatzas2014brownian} it suffices\footnote{We here state stronger and easier-to-state hypotheses than in \cite[Chapter 5.2, Theorem 2.5]{karatzas2014brownian}, which suffice for our purposes.} to show the Lipschitz property that there exists $K$ for which
    \begin{equation}\label{eq:Lipschitz}
        ||\bm{b}(T,\bm{x}) - \bm{b}(T,\bm{y})|| + ||\sigma(T,\bm{x})-\sigma(T,\bm{y})|| \leq K ||\bm{x}-\bm{y}||.
    \end{equation}
    Here the norm is the standard Euclidean one, viewing $\sigma$ as a vector in $\R^{n^2}$.
    For strong existence, by \cite[Chapter 5.2, Theorem 2.9]{karatzas2014brownian} it suffices to show \eqref{eq:Lipschitz} in addition to 
    \begin{equation}
        \label{eq:lin_growth}
        ||\bm{b}(T,\bm{x})||^2 + ||\sigma(T,\bm{x})||^2 \leq K^2 ( 1+||\bm{x}||^2).
    \end{equation}
    A crude bound shows 
    \[
||\bm{b}(T,\bm{x})||^2 \leq 4n^3 ||x||^2.
    \]
    Take $K^2=4n^3$. Since $\sigma$ is constant and $\bm{b}(T,\bm{x})$ is linear in $\bm{x}$, \eqref{eq:Lipschitz} holds. Since $||\sigma(T,\bm{x})||^2 = n$, \eqref{eq:lin_growth} holds as well, completing the proof.
\end{proof}

We now find that the explicit Gaussian vector derived in \Cref{thm:aux_limit_and_nicer_covariances} describes the stationary distribution of the above system of SDEs.

\begin{proposition}\label{thm:stationarity}
Let $(Z^{(1)}_T,Z^{(2)}_T,\ldots)$ be the vector-valued stochastic process satisfying the system of SDEs
\begin{equation}
    dZ^{(k)}_T = \left((k-1)Z^{(k-1)}_T - k Z^{(k)}_T\right)dT + dW^{(k)}_T
\end{equation}
where $W^{(k)}_T$ are independent standard Brownian motions, with initial distribution $(Z^{(k)}_0)_{k \geq 1}$ given by a Gaussian vector with covariances 
\[
\Cov(Z^{(r)}_0,Z^{(s)}_0) = \frac{1}{4 \pi^2} \displaystyle \oint_{\Gamma_0} \oint_{\Gamma_{0,w}} \frac{w}{z-w} \frac{r! s!}{z^{r}w^{s}}e^{z+w}(1-z/r)(1-w/s)\frac{dz}{z} \frac{dw}{w}.
\]
Then $(Z^{(k)}_T)_{k \geq 1}$ is stationary, i.e.  
\begin{equation}\label{eq:stationary_in_dist}
    (Z^{(k)}_{T_0})_{k \geq 1} = (Z^{(k)}_0)_{k \geq 1} 
\end{equation}
in distribution, for any fixed time $T_0 > 0$.
\end{proposition}

\begin{remark}
A natural further question is whether the finite-$\tau$ SDEs \eqref{eq:X_SDE} admit a Gaussian solution with fixed-time covariances given by our finite-$\tau$ formula in \Cref{thm:gaussianity}. This seems more difficult to address without the large-time simplification of \Cref{thm:aux_limit_and_nicer_covariances}, and we have not attempted to pursue it in this work.
\end{remark}

To prepare for the proof, we first give two computational lemmas, which will be proven at the end of the section.

\begin{definition}
For $r,s \in \Z_{\geq 1}$, let
\begin{equation}\label{eq:defD}
D(r,s) := \frac{1}{4 \pi^2} \displaystyle \oint_{\Gamma_0} \oint_{\Gamma_{0,w}} \frac{w}{z-w} \frac{r! s!}{z^{r}w^{s}}e^{z+w}(1-z/r)(1-w/s)\frac{dz}{z} \frac{dw}{w}.
\end{equation}
\end{definition}
By \Cref{thm:aux_limit_and_nicer_covariances}, when $r \geq s$ one has $D(r,s) = \Cov(\zeta_r,\zeta_s)$, but as noted in \Cref{rmk:formula_not_symmetric} this is not true when $r < s$. This will be important in computations below.

\begin{lemma}\label{thm:integral_is_0}
For any $r \geq s \geq 1$, 
\begin{equation}\label{eq:integral_is_0}
    (r-1)D(r-1,s) + (s-1)D(r,s-1) - (r+s) D(r,s) = 0.
\end{equation}
\end{lemma}

\begin{lemma}\label{thm:integral_is_1}
For any $r \geq 2$, 
\begin{equation}\label{eq:integral_is_1}
    D(r-1,r)-D(r,r-1) = \frac{1}{r-1}.
\end{equation}
\end{lemma}

\begin{proof}[Proof of {\Cref{thm:stationarity}}]
It suffices to show
\begin{equation}\label{eq:finite_stationarity}
    (Z^{(k)}_{T_0})_{1 \leq k \leq n} = (Z^{(k)}_0)_{1 \leq k \leq n} \quad \quad \text{in distribution}
\end{equation}
for each $n \geq 1$ and $T_0>0$. First note that the solution $(Z^{(k)}_T)_{1 \leq k \leq n}$ is a Gaussian process, so its distribution at time $T_0$ is determined by its covariance matrix, i.e. it suffices to check
\begin{equation}\label{eq:covs_equal}
    \Cov\left(Z^{(r)}_{T_0},Z^{(s)}_{T_0}\right) = \Cov\left(Z^{(r)}_0,Z^{(s)}_0\right)
\end{equation}
for each $1 \leq s \leq r$. Let
\[
A_{r,s}(T) := \Cov\left(Z^{(r)}_{T},Z^{(s)}_{T}\right)
\]
for $r,s \geq 1$. It follows by applying It\^o's  lemma that 
\begin{equation}\label{eq:cov_ODE}
    \dderiv{}{T} A_{r,s}(T) = \bbone(r=s) + (r-1) A_{r-1,s}(T) + (s-1)A_{r,s-1}(T) - (r+s)A_{r,s}(T) 
\end{equation}
(this computation can be done for quite general systems of SDEs, see \cite[(4.3)]{borodin2017stochastic}). Hence to check \eqref{eq:covs_equal}, it suffices to check that the RHS of \eqref{eq:cov_ODE} is $0$ when the constant solution 
\[
A_{r,s}(T) =  D(r,s)
\]
is plugged in. When $r > s$, this follows directly from \Cref{thm:integral_is_0}. When $r=s$, since 
\[
A_{r-1,r} = A_{r,r-1} = D(r,r-1)
\]
we have
\[
\text{RHS{\eqref{eq:cov_ODE}}} = 1 + (r-1)D(r-1,r) + (r-1)D(r,r-1) - 2r D(r,r) + (r-1)(D(r,r-1) - D(r-1,r))
\]
which is $0$ by \Cref{thm:integral_is_0} and \Cref{thm:integral_is_1}. This completes the proof.
\end{proof}

\begin{proof}[Proof of {\Cref{thm:integral_is_0}}]
We obtain that 
\begin{equation}\label{eq:exp_to_power}
\frac{1}{(2 \pi i)^2}  \oint_{\gamma_0} \oint_{\gamma_{0,w}} \frac{w}{z-w}e^{z+w}z^{-a}w^{-b} \frac{dz}{z} \frac{dw}{w} = \frac{1}{(2 \pi i)^2}  \oint_{\gamma_0} \oint_{\gamma_{0,w}} \frac{w}{z-w}\frac{(z+w)^{a+b}}{(a+b)!} z^{-a}w^{-b} \frac{dz}{z} \frac{dw}{w} 
\end{equation}
for $a,b \geq 0$, by expanding 
\[
\frac{w}{z-w} = \pfrac{w}{z} + \pfrac{w}{z}^2+\ldots 
\]
(using that $|w|<|z|$ along the contours) and taking the residue expansion of both sides. Using \eqref{eq:exp_to_power} to convert the integral in \eqref{eq:defD} to one with integrand of the form 
\[
\frac{w}{z-w}\text{(Laurent polynomial in $z,w$)},
\]
and combining the three integrals in \eqref{eq:integral_is_0} into a single double contour integral, it is easily verified (by a computer) that the integrand is $0$.
\end{proof}

\begin{proof}[Proof of {\Cref{thm:integral_is_1}}]
One has
\begin{multline*}
    D(r-1,r)-D(r,r-1) = \frac{1}{4 \pi^2} \oint_{\gamma_0} \oint_{\gamma_{0,w}} \frac{w}{z-w} (r-1)!(r-2)!e^{z+w} \\
    \cdot \left(\frac{1}{z^{r-1}w^r}(r-1-z)(r-w)+\frac{1}{z^rw^{r-1}}(r-z)(r-1-w)\right)\frac{dz}{z}\frac{dw}{w}.
\end{multline*}
Since
\[
\left(\frac{1}{z^{r-1}w^r}(r-1-z)(r-w)+\frac{1}{z^rw^{r-1}}(r-z)(r-1-w)\right) = \frac{1}{z^rw^r}(w-z)((r-z)(r-w)-r)
\]
which cancels the $\frac{1}{z-w}$ in the integrand, the only poles in the integrand are at $w=0$ and $z=0$. The result \eqref{eq:integral_is_1} now follows by taking this residue.
\end{proof}

\begin{proof}[Proof of \Cref{thm:auxlimit_intro}]
Uniqueness of the stationary solution to \eqref{eq:SDEs_intro} was shown in \Cref{thm:SDE_uniqueness}. In \Cref{thm:aux_limit_and_nicer_covariances} we showed that the $X^{(i)}_\tau$ converge to a jointly Gaussian vector with the explicit covariances given in \Cref{thm:auxlimit_intro}, which is the second half of the theorem. In \Cref{thm:stationarity} we showed that this jointly Gaussian vector also describes the unique stationary solution to \eqref{eq:SDEs_intro}, which accounts for the first half.
\end{proof}

\section{Bulk fluctuations}\label{sec:bulk}

In this section, we gather the random variables $\zeta_i$ into a single stochastic process, and compute its covariance in \Cref{thm:bulk_intro} by analysis of the contour integral from \Cref{thm:aux_limit_and_nicer_covariances}.

\begin{definition}
Let $Y_T, T \in \R^+$ be the stochastic process for which $Y_0 = 0$, $Y_n = \zeta_n$ for all $n \in \Z_{\geq 1}$ with $\zeta_n$ as in \Cref{thm:aux_limit_and_nicer_covariances}, and 
\[
Y_{n+\alpha} = (1-\alpha )Y_n + \alpha Y_{n+1}
\]
for $n \in \Z_{\geq 1}, \alpha \in (0,1)$.
\end{definition}

Finally, we recall the main result.

\bulkintro*

\begin{proof}

Before getting to the main computation, we must take care of some technical details. Firstly, we are justified in speaking of the unique Gaussian process with covariances as in the theorem statement, because a Gaussian process is determined by its (jointly Gaussian) finite-dimensional distributions, and these Gaussian vectors are determined by their covariances. $R_T^{(k)}$ is a Gaussian process; when $k + T\sqrt{k} \in \Z$, $R_T^{(k)} = k^{1/4}\zeta_{k + T\sqrt{k}}$ is Gaussian, and for other values of $T$ $R_T^{(k)}$ is a convex combination of Gaussians and hence also Gaussian. Hence to show convergence of finite-dimensional distributions to $R_T$, it suffices to show convergence of pairwise covariances of $R_T^{(k)}$ to those of $R_T$, i.e. we must show
\begin{equation}\label{eq:cov_conv_sufficient}
    \Cov(R_a^{(k)}, R_b^{(k)}) \to \int_0^\infty y^2 e^{-y^2-|b-a|y}dy \text{     as $k \to \infty$,}
\end{equation}
where without loss of generality $a \geq b$. 

Since $R_a^{(k)}$ is in general a convex combination $p(a,k)\zeta_{k+\floor{a\sqrt{k}}} + (1-p(a,k))\zeta_{k+\ceil{a \sqrt{k}}})$ with some $p(a,k) \in [0,1]$, and similarly for $R_b^{(k)}$, to show \eqref{eq:cov_conv_sufficient} it suffices to show
\begin{equation}\label{eq:cov_conv_with_zeta}
    k^{1/2}\Cov(\zeta_{k+\floor{a\sqrt{k}}}, \zeta_{k+\floor{b\sqrt{k}}}) \to \int_0^\infty y^2 e^{-y^2-|b-a|y}dy \text{     as $k \to \infty$,}
\end{equation}
along with the same convergence where one or both floor functions are replaced by ceiling functions. We will show \eqref{eq:cov_conv_with_zeta} by steepest-descent analysis of the integral formula for covariances \eqref{eq:final_prelimit_cov}, and the versions with one or both floor functions replaced by ceilings are exactly analogous.

Let $r = k + \floor{a \sqrt{k}}, s = k + \floor{b \sqrt{k}}$.

First change variables in \eqref{eq:final_prelimit_cov} to $\tz = z/r, \tw = w/s$ to obtain
\begin{equation}
    \sqrt{k}\Cov(\zeta_r, \zeta_s) =\frac{1}{4 \pi^2} \displaystyle \oint_{\Gamma_0} \oint_{\Gamma_{0,\frac{s}{r}\tw}} \sqrt{k} \frac{s \tw}{r \tz-s \tw} \frac{r! s!}{r^{r}s^{s} \tz^{r}\tw^{s}}e^{r\tz+s\tw}(1-\tz)(1-\tw) \frac{d\tz}{\tz}\frac{d\tw}{\tw}
\end{equation}
Using Stirling's approximation $n! = \sqrt{2 \pi n}(n/e)^n e^{o(1)}$, the above equals
\begin{equation}\label{eq:int_after_stirling}
    \frac{1}{4 \pi^2} \displaystyle \oint_{\Gamma_0} \oint_{\Gamma_{0,\frac{s}{r}\tw}}\sqrt{k} (2 \pi \sqrt{rs})\frac{ s \tw}{r \tz-s \tw}(1-\tz)(1-\tw) e^{rF(\tz)+sF(\tw)+o(1)} \frac{d\tz}{\tz}\frac{d\tw}{\tw},
\end{equation}
where here and henceforth $F(z) = z-\log z - 1$. It is easy to check that $F(z)$ has a unique critical point at $z=1$ which is second-order, and our steepest descent will consist of zooming in on this critical point.

For the contours $\Gamma_0$ and $\Gamma_{0,\frac{s}{r}\tw}$ above, we will use the (counterclockwise-oriented) contours $C_\tz = C_\tz(k)$ and $C_\tw = C_\tw(k)$ which are pictured in \Cref{fig:contours} and which we now describe. First fix any $\delta$ with $1/3 < \delta < 1/2$. Let 
\begin{align}
    C_\tz &= \{\min(x,1)+iy: x^2+y^2= 1+k^{-2\delta}\} \\
    C_\tw &= \{\min(x,1-k^{-1/2}) + i y: x^2+y^2 = (1-k^{-1/2})^2+k^{-2\delta}\}.
\end{align}
Each contour has two parts, one a subset of a circle and one a vertical line; call the circular parts $C_\tz'',C_\tw''$ and the vertical parts $C_\tz', C_\tw'$. 

\begin{figure}[h]
\begin{center}
\includegraphics{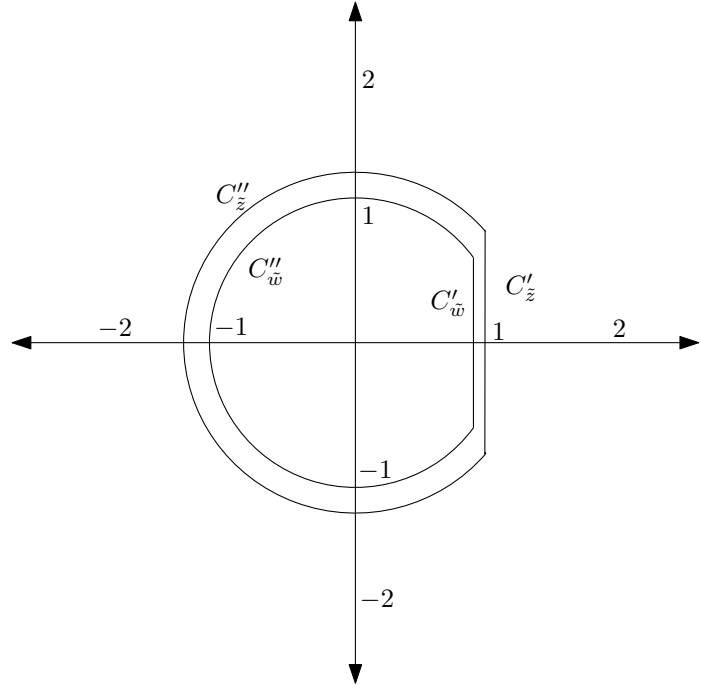}
\end{center}
 \caption{The contours $C_\tz$ and $C_\tw$ (figure not to scale).}\label{fig:contours}
\end{figure}

We first claim that only the integral with $(\tw,\tz) \in C_\tw' \times C_\tz'$ contributes asymptotically, i.e.
\begin{equation}\label{eq:int_vanishes_outside_vertical}
    \oint \oint_{(C_\tw \times C_\tz) \setminus (C_\tw' \times C_\tz')} \sqrt{k} (2 \pi \sqrt{rs})\frac{ s \tw}{r \tz-s \tw}(1-\tz)(1-\tw) e^{rF(\tz)+sF(\tw)+o(1)} \frac{d\tz}{\tz}\frac{d\tw}{\tw} \to 0
\end{equation}
as $k \to \infty$.

It is clear from the definition of the contours that the distance between them is $\const \cdot k^{-1/2} + o(k^{-1/2})$. Since $r \geq s$, we therefore have 
\begin{equation}
    \frac{1}{r\tz-s\tw} \leq \const \cdot  k^{1/2}/s.
\end{equation}
Hence 
\begin{equation}
    \abs*{\sqrt{k} (2 \pi \sqrt{rs})\frac{ s \tw}{r \tz-s \tw}\frac{(1-\tz)(1-\tw)}{\tz \tw}} \leq \text{const}  \cdot k^{2}
\end{equation}
for some constant, for all large enough $k$.
We have 
\begin{equation*}
    \sup_{\tz \in C_\tz''}\Re (F(\tz)) \leq -\log|\tz| =  -\log \sqrt{1+k^{-2\delta}} = -\frac{1}{2}k^{-2\delta}+o(k^{-2\delta})
\end{equation*}
and similarly 
\begin{equation*}
\sup_{\tw \in C_\tw''}\Re F(\tw) \leq -\frac{1}{2}k^{-2\delta}+o(k^{-2\delta}).    
\end{equation*}
Hence for such $\tz \in C_\tz''$ and $\tw \in C_\tw''$,  we have $\Re(r F(\tz)), \Re(s F(\tw)) \leq -\frac{1}{2}k^{1-2\delta} + o(k^{1-2\delta})$. On the vertical segments $C_\tz'$ and $C_\tw'$, $\Re F$ is maximized at the unique real value, so $\sup_{\tz \in C_\tz'} \Re F(\tz) = 0$ and 
\begin{equation}
    \sup_{\tw \in C_\tw'} \Re F(\tw) = - k^{-1/2} - \log (1-k^{-1/2}) = O(k^{-1}).
\end{equation}
Thus if $\tz \in C_\tz, \tw \in C_\tw$ and at least one of $\tz \in C_\tz''$ or $\tw \in C_\tw''$ holds, 
\begin{equation}
    e^{rF(\tz)+sF(\tw)}\leq  e^{-\frac{1}{2}k^{1-2\delta} + o(k^{1-2\delta})}.
\end{equation}

It follows that the integrand in \eqref{eq:int_vanishes_outside_vertical} is bounded by 
\begin{equation}
    \const \cdot  k^{2} e^{-\frac{1}{2}k^{1-2\delta} + o(k^{1-2\delta})}
\end{equation}
uniformly in $k$ over the domain of integration (which, recall, also depends on $k$). Since the lengths of the $k$-dependent contours $C_\tz,C_\tw$ are bounded over all $k$, and the above bound converges to $0$ since $1-2\delta > 0$, we have established \eqref{eq:int_vanishes_outside_vertical}.

Now we consider the remaining part of the integral,
\begin{equation}\label{eq:int_on_vertical_before_uv_sub}
    \frac{1}{2 \pi} \oint_{C_\tw'} \oint_{C_\tz'}\sqrt{krs}\frac{ s \tw}{r \tz-s \tw}(1-\tz)(1-\tw) e^{rF(\tz)+sF(\tw)+o(1)} \frac{d\tz}{\tz}\frac{d\tw}{\tw}
\end{equation}
We have $F'(\tz) = 1-1/\tz$ and $F''(\tz) = 1/\tz^2$, so Taylor expanding about $1$ we have $F(\tz) = (\tz-1)^2/2 + O((\tz-1)^3)$. Since $C_\tz' = \{1+iy: y \in (-k^{-\delta},k^{-\delta})\}$, for $\tz \in C_\tz'$ one has $|\tz-1|^3 < k^{-3\delta}$, and similarly $|\tw-1|^3 < k^{-3\delta}$ for $\tw \in C_\tw'$. Because $\delta > 1/3$, we have $|F(\tz) - (\tz-1)^2/2| = o(1)$ as $k \to \infty$ uniformly over $C_\tz'$, and similarly for $\tw$. 

Let us change variables in \eqref{eq:int_on_vertical_before_uv_sub} to $u,v$, defined by $\tz = 1+i k^{-1/2} u$ and $\tw = 1-k^{-1/2}+i k^{-1/2} v$. The condition that $z \in C_\tz', w \in C_\tw'$ translates to $-k^{1/2-\delta} < u,v < k^{1/2-\delta}$, and by the previous paragraph
\begin{align*}
    F(\tz) &= -\frac{1}{2k}u^2 + o_u(1) \\
    F(\tw) &= -\frac{1}{2k}(v+i)^2 + o_v(1)
\end{align*}
where the error terms depend on $u$ (resp. $v$) but are bounded uniformly on the domain of integration. Thus we may write \eqref{eq:int_on_vertical_before_uv_sub} as
\begin{align}\label{eq:int_after_uv_sub}
\begin{split}
    &\frac{1}{2\pi} \int_\R \int_\R \bbone(|u|,|v| < k^{1/2-\delta})\sqrt{krs} \frac{s(1+k^{-1/2}(-1+iv))}{(\sqrt{k}+\floor{a\sqrt{k}}-\floor{b\sqrt{k}}) + \sqrt{k} i u - \sqrt{k} i v + o(\sqrt{k})} \\
    &\quad \quad \times \left(-i\frac{u}{\sqrt{k}}\right)\left(\frac{1-iv}{\sqrt{k}}\right) e^{-u^2/2-(v+i)^2/2 + o_{u,v}(1)} \frac{i k^{-1/2}du}{1+i u /\sqrt{k}} \frac{i k^{-1/2}dv}{1+k^{-1/2}(-1+iv)} \\
    &= \frac{1}{2\pi} \int_\R \int_\R \bbone(|u|,|v| < k^{1/2-\delta}) \frac{s\sqrt{krs}}{k^{5/2}} \frac{1+k^{-1/2}(-1+iv)}{1+a-b+iu-iv + o(\sqrt{k})} \\
    & \quad \quad \times e^{-u^2/2-(v+i)^2/2 + o_{u,v}(1)}\frac{u(i+v)}{(1+k^{-1/2}(-1+iv))(1+i u /\sqrt{k})} du dv
\end{split}
\end{align}
Recalling that $r$ and $s$ are $k+o(k)$ and $u,v = o(k^{1/2})$ in the domain of integration, we see that the integrand in \eqref{eq:int_after_uv_sub} converges to 
\begin{equation}
    \frac{1}{1+(a-b)+iu-iv} u(v+i) e^{-u^2/2 - (v+i)^2/2} du dv
\end{equation}
as $k \to \infty$, and furthermore that there exists a constant $C$ such that it is dominated by the integrable function 
\begin{equation}
    C \bbone\left(|u|,|v| < k^{1/2-\delta}\right) \frac{1}{c+(a-b)+iu-iv} u(v+ic) e^{-u^2/2 - (v+i)^2/2}
\end{equation}
for all $u,v \in \R$ and all large enough $k$. Hence by dominated convergence, \eqref{eq:int_after_uv_sub} converges to 
\begin{equation}
    \frac{1}{2\pi}\int_\R \int_\R \frac{1}{1+(a-b)+iu-iv} u(v+i) e^{-u^2/2 - (v+i)^2/2} du dv.
\end{equation}
Note that the convergence above would be exactly the same if $\floor{a\sqrt{k}}$ and/or $\floor{b\sqrt{k}}$ had been replaced by ceiling functions, as mentioned earlier. 

Using the identity 
\begin{equation}
\frac{1}{\alpha} = \int_0^\infty e^{-y\alpha} dy
\end{equation}
if $\Re(\alpha) > 0$, since $1+ a-b \geq 1 $ the above integral is equal to 
\begin{align}\label{eq:split_into_u_and_v_integral}
    &\frac{1}{2\pi} \int_\R \int_\R \left(\int_0^\infty e^{-y(1+a-b+iu-iv)}dy\right)u(v+i) e^{-u^2/2 - (v+i)^2/2} du dv \\
    &= \frac{1}{2\pi} \int_0^\infty e^{-y^2-y(a-b)} \left(\int_\R e^{-\frac{1}{2}(u^2+2iyu-y^2)}du\right)\left(\int_\R e^{-\frac{1}{2}((v+i)^2-2iy(v+i)-y^2)}dv\right)dy.
\end{align}
Since 
\begin{equation*}
    \int_\R e^{-\frac{1}{2}(u^2+2iyu-y^2)}du = -\sqrt{2\pi} iy 
\end{equation*}
and 
\begin{equation*}
    \int_\R e^{-\frac{1}{2}((v+i)^2-2iy(v+i)-y^2)}dv = \sqrt{2\pi} i y
\end{equation*}
(for the latter we must shift contours from $\R$ to $\R-i$ before evaluating the Gaussian integral), we have that \eqref{eq:split_into_u_and_v_integral} is equal to 
\begin{equation}
    \int_0^\infty y^2 e^{-y^2-y(a-b)}dy,
\end{equation}
completing the proof.
\end{proof}






\end{document}